\documentclass[10pt,a4paper, reqno]{amsart}
\usepackage{mathrsfs}

\usepackage{amsmath,amsfonts,verbatim}
\usepackage{latexsym}
\usepackage{amssymb,leftidx}
\usepackage{extarrows}
\usepackage{overpic}
\usepackage{color}
\usepackage{epsfig}
\usepackage{subfigure}
\usepackage{tikz}
\usepackage{bbm}
\usepackage{tikz, caption}
\usepackage{physics}
\usepackage[font=small,labelfont=bf]{caption}

\usepackage[colorlinks=true,backref=page]{hyperref}
\hypersetup{urlcolor=blue, citecolor=blue}

\usepackage{amssymb}
\usepackage{mathrsfs}
\usepackage{amscd}
\usepackage{cite}

\newcommand\R{\mathbb{R}}

\usepackage{graphicx}
\usepackage{float}

\numberwithin{equation}{section}
\newtheorem{proposition}{Proposition}[section]
\newtheorem{definition}{Definition}[section]
\newtheorem{lemma}{Lemma}[section]
\newtheorem{assumption}{Assumption}[section]
\newtheorem{theorem}{Theorem}[section]
\newtheorem{corollary}{Corollary}[section]
\newtheorem{remark}{Remark}[section]

\begin{document}

\title[Stability and control]{Control and stabilization problem for a class of fourth-order nonlinear Schr\"odinger equation on boundaryless compact manifold }

\author{Yilin Song}
\address{Yilin Song
\newline \indent The Graduate School of China Academy of Engineering Physics, Beijing 100088,\ P. R. China}
\email{songyilin21@gscaep.ac.cn}

\author{Jiqiang Zheng}
\address{Jiqiang Zheng
\newline \indent Institute of Applied Physics and Computational Mathematics, Beijing, 100088, China.
		\newline\indent
		National Key Laboratory of Computational Physics, Beijing 100088, China}
\email{zheng\_jiqiang@iapcm.ac.cn, zhengjiqiang@gmail.com}	

\author{Ruihan Zhou}
\address{Ruihan Zhou
\newline \indent Institute of Applied Physics and Computational Mathematics, Beijing, 100088, China.}
\email{19210180085@fudan.edu.cn}

\begin{abstract}
In this paper, we study the stabilization property and large time controllability  for a class of fourth-order Schr\"odinger equations on a compact manifold without boundary in dimensions $1\leq d\leq5$:
\begin{align}
i\partial_tu+(\Delta_g^2-\beta\Delta_g)u=-|u|^{2k}u, \,\,x\in M\tag{4NLS}\label{4NLS1}
\end{align}
where $k\in\mathbb{N}$ and $\beta\in\mathbb{R}_{+}$ when $1\leq d\leq 4$ but $\beta\in\mathbb{Q}_{+}$ for $d=5$.   We adapt the strategy in  Macia [Vietnam J. Math. (2021)] to establish observability and the propagation of singularities.  Moreover, we use these propagation estimates to deduce the unique continuation property for $\eqref{4NLS1}$. By the classical Hilbert Uniqueness Method (HUM) and the Picard iteration,  the stabilization and large time controllability  hold  under the Geometric Control Condition (GCC) and the Unique Continuation Property (UCP) for the linearized equation.  To obtain the controllability and stabilization at the $H^2$ level  with $d=5$,  we will focus on $M=\Bbb S^5$ with $k=1$. Our results extend those of  Laurent [SIAM J. Math. Anal. (2009)] and Capistrano Filho-Pampu [Math. Z., 2022].
\end{abstract}
	
\maketitle
	
\begin{center}
\begin{minipage}{100mm}
\end{minipage}
\end{center}





\section{Introduction}
In this paper, we study the global stabilization and exact controllability of the fourth-order nonlinear Schr\"odinger equation 
\begin{equation}\label{4NLS}
\begin{cases}
i\partial_tu+(\Delta_g^2-\beta\Delta_g)u=-|u|^{2k}u,&(t,x)\in\R\times M,\\
u(0,x)=u_0(x),& ~~~~~~x\in M,
\end{cases}
\end{equation}
posed on a $d$-dimension compact  Riemannian manifold without boundary, where the parameter $\beta\in\mathbb{R}^{+}$ for $1\leq d\leq4$, $\beta\in\mathbb{Q}^{+}$ for $d=5$,
and $u:I\times M\to\mathbb{C}$ is the unknown complex-valued function. Here $\Delta_{g}$ denotes the Laplace-Beltrami operator on $M$ associated with the Riemannian metric $g$.  We write $\Delta_{g}$ as follows:
\begin{equation}
\Delta_{g}=\frac{1}{\sqrt{|g|}}\sum_{i,j=1}^{d}\frac{\partial}{\partial x_{i}}\Big(g^{ij}\sqrt{|g|}\frac{\partial}{\partial x_{j}}\Big).
\end{equation}
The  equation \eqref{4NLS} possesses a Hamiltonian structure and preserves both the mass and the energy
\begin{equation}
\operatorname{Mass}: M(u)(t)=\int_{M}|u|^{2}\,\dd x,
\end{equation}
\begin{equation}
\operatorname{Energy}: E(u)(t)=\int_{M}\Big(\frac{1}{2}|\Delta_{g}u|^{2}+\frac\beta2|\nabla_{g} u|^{2}+\frac{1}{2k+2}|u|^{2k+2}\Big)\,\dd x.
\end{equation}

In the past few decades, there has been a surge of interest in applying the fourth order Schr\"odinger equation to characterize physical phenomena, and the equation has been extensively studied in the literature \cite{Karpman1,Karpman2,MXZ,Pausader1,Pausader2}.  It is known that fourth order Schr\"odinger equations  are not merely formal generalizations of second order counterparts, but also arise naturally from intrinsic physical mechanisms, such as, relativistic quantum corrections, dipolar interactions in Bose-Einstein condensation, and high-order dispersion in nonlinear optics. On the other hand, many physicists and
mathematicians are devoted to studying the Gross-Pitaevskii equation in which the confining potential is added to the fractional Laplacian operator. The confining effect
brought by the potential is close to the localization effect on the compact manifold.
Thus, the study of  higher order dispersive equations in the spatially inhomogeneous media motivates us
to consider defocusing fourth order nonlinear Schr\"odinger equation (4NLS) on the compact
manifold $M$.

In this article, we focus on the controllability and stabilization problems associated with the Cauchy problem  \eqref{4NLS}.
\begin{itemize}
\item[$\bullet$]The controllability problem asks whether one can steer the system to a prescribed terminal state $y(T)=y_{1}$ by means of an appropriate control input $h$ belonging to a control space $U_{T}$. More precisely, suppose that $H$ is a Hilbert space, and the linear abstract Cauchy problem
\begin{equation}\label{Abstract Cauchy problem}
\begin{cases}
\frac{\dd}{\dd t}y=Ay+Bh, &t\in[0,T]\\ y(0)=y_{0}\in H 
\end{cases}
\end{equation}
is well-posed on $H$. Given a final state $y_{1}\in H$, can we find a control function $h\in U_{T}$ such that the solution of the Cauchy problem \eqref{Abstract Cauchy problem} reaches the target state $y(T)=y_{1}$? A classical approach to this problem is the Hilbert Uniqueness Method (HUM) introduced by J. Lions \cite{Lions}, which converts the controllability problem  into an observablility inequality for the adjoint system. This method has become a crucial tool in control theory for dispersive equations.
	
\item[$\bullet$]The stabilization problem, tracing back to Lyapunov's fundamental work, may be viewed as a particular manifestation of controllability. Its objective is to design a feedback mechanism ensuring that the trajectory of a dynamical system converges to the zero solution. That is to say, we search for a control function $h$ such that the solution $y(t)\to 0$ in $H$ as $t\to\infty$. A typical example arises in the study of the wave equation, where internal damping of the form $u_{t}$ is employed as a feedback term. Under appropriate geometric conditions, this yields energy decay; see Burq and G\'erard \cite{Burqdamping}.
\end{itemize}

We have now outlined the basic concept of controllability and stabilization. It is well-known that establishing the well-posedness constitutes the fundamental step for investigating these problems. Accordingly, to lay a foundation for the analysis of its controllability and stabilization for \eqref{4NLS}, we first need to establish the well-posedness of the following system:
\begin{equation}
\begin{cases}
i\partial_{t}u+(\Delta_{g}^{2}-\beta\Delta_{g})u=-|u|^{\alpha-1}u+1_{\omega}h(t,x),&(t,x)\in\mathbb{R}\times M,\\ u(0,x)=u_{0}(x),& x\in M.
\end{cases}
\end{equation}

In recent years, dispersive equations on non-flat geometries have attracted substantial attention. The first contribution in this area comes from Kapitanski's seminal work \cite{Kapitanski}, where the WKB construction was initially proposed to establish Strichartz estimates for the wave equation. In contrast, Strichartz estimates for the linear Schr\"odinger equation are more complicated due to the infinite propagation speed. This novel and nontrivial behavior was identified by Burq, G\'erard, and Tzvetkov \cite{Burq-AJM} via the semiclassical analysis. Specifically, they established the following inequality:
\begin{equation}\label{strichartz on manifolds}
\|e^{it\Delta_{g}}f\|_{L_{t}^{p}(I,L_{x}^{q}(M))}\leq C(I)\|f\|_{H^{\frac{1}{p}}},~\frac{2}{p}=d\Big(\frac{1}{2}-\frac{1}{q}\Big),~p\geq2,~q<\infty.
\end{equation}

As an application of Strichartz estimates, Burq-G\'erard-Tzvetkov \cite{Burq-AJM} established the local well-posedness for the cubic nonlinear Schr\"odinger equation with initial data $u_0\in H^{s}$ for $s>\frac{d-1}{2}$. Later, Dinh \cite{Dinh} generalized  \eqref{strichartz on manifolds} to the fractional case $\sigma\in(1,\infty)$:
\begin{equation}
\|e^{it(-\Delta_{g})^{\frac{\sigma}{2}}}f\|_{L^{p}_{t}L^{q}_{x}(I\times M)}\leq C(I)\|f\|_{H^{\frac{1}{p}}},~ \frac{2}{p}\leq d\big(\frac{1}{2}-\frac{1}{q}\big),~ p\geq2, ~q<\infty.
\end{equation}
Then, he obtained the well-posedness of the cubic fractional Schr\"odinger equation in $H^s(M)$ with $s>\frac{d-1}{2}$.

Building on the above well-posedness results established in \cite{Burq-AJM}, Dehman-G\'erard-Lebeau \cite{Dehman} were the first to obtain exact controllability and stabilization results for the Schr\"odinger equation on compact manifolds without boundary. They proceed  by reducing the local exact controllability problem to an observability inequality and the stabilization problem to the energy decay estimate for the damped Schr\"odinger equation. Combining these ingredients, they derived global exact controllability. Later, Laurent \cite{Laurent} utilized the multilinear Strichartz estimate established in \cite{Burq-Multilinear} to prove stabilization for the cubic NLS on several three-dimensional manifolds including $\mathbb{T}^3$, $\mathbb{S}^3$ and $\mathbb{S}^2\times\mathbb{S}^1$. Very recently, applying the partial Floquet transformation, Niu-Zhao \cite{Niu-Zhao} proved the controllability for the cubic NLS on semi-periodic space. Following the approach of \cite{Dehman} and using  Strichartz estimates proved in \cite{Dinh},  Capistrano-Filho and Pampu \cite{Robertomathz} derived  the controllability and stabilization results for the fractional Schr\"odinger equation.  For fourth order equations, Capistrano-Filho and Cavalcante \cite{Robertoamo} established the controllability and stabilization in $L^{2}(\mathbb{T})$ for equation
\begin{equation*}
i\partial_{t}u+\partial_{x}^{2}u-\partial_{x}^{4}u=\lambda |u|^{2}u.
\end{equation*}
We also refer to \cite{Cavalcanti,Coron,Coron-Krieger-Xiang,Dehman-Lebeau-Zuazua,Krieger-Xiang,Krieger-Xiang2,Laurent-JFA,Laurent-Joly} for the results on controllability of nonlinear wave equation and Klein-Gordon equation.
\subsection{Main results}
Our goal in this paper concerns the controllability and stabilization problems for equation \eqref{4NLS}.   Before stating our main results, we state the main assumptions.
\begin{assumption}[Geometric Control Condition]\label{GCC}
   There exists $T_{0}>0$ such that every geodesic of $M$,   starting at $t=0$ with speed $1$, entering the set $\omega$ in a time $t<T_{0}$.
\end{assumption}
\begin{remark}
On $\mathbb{S}^d$, it suffices to take $\omega$ as a neighborhood of the equator $\{x_6=0\}$, which always satisfies the Geometric Control Condition.
\end{remark}
\begin{theorem}[Controllability]\label{thm1}
Let $d\leq4$, $k\in\mathbb{N}$, $\beta\in\mathbb{R}^{+}$, and let the subset $\omega\subset M$ satisfy the Geometric Control Condition (GCC). Then for any $R_{0}>0$, there exist $T(R_{0})>0$, $C>0$ such that initial data $u_0$ and final data $u_1$ with
\begin{equation}
		\|u_{0}\|_{H^{2}(M)}\leq R_{0}~~ \textit{and}~~\|u_{1}\|_{H^{2}(M)}\leq R_{0},
	\end{equation}
there exists a control $h\in L^{2}([0,T]; H^{2}(M))$ with $\|h\|_{L^{2}([0,T]; H^{2}(M))}\leq C$  such that the solution $u\in C([0,T],H^{2}(M))$ of 
\begin{equation}\label{equ-1}
\begin{cases}
i\partial_{t}u+\Delta_{g}^{2}u-\beta\Delta_{g}u+|u|^{2k}u=1_{\omega}h(t,x),&(t,x)\in[0,T]\times M\\ u(0,x)=u_{0}(x)\in H^{2}(M)
\end{cases}
\end{equation}
satisfies $u(T)=u_{1}$.
	
Moreover, when $d=5$, $\beta\in\mathbb{Q}^+$, the controllability result also holds for the Cauchy problem
\begin{equation}\begin{cases}
i\partial_{t}u+\Delta_{g}^{2}u-\beta\Delta_{g}u+|u|^{2}u=1_{\omega}h(t,x),&(t,x)\in[0,T]\times\mathbb{S}^{5},\\
u(0,x)=u_{0}(x)\in H^{2}(\mathbb{S}^{5}),
\end{cases}\end{equation}
in the following sense: for any $R_{0}>0$, there exist $T(R_{0})>0$ and $C>0$ such that for every $u_{0},u_{1}\in H^{2}(\mathbb{S}^{5})$ with
$$\|u_{0}\|_{H^{2}(\mathbb{S}^{5})}\leq R_{0}\quad\text{and}\quad\|u_{1}\|_{H^{2}(\mathbb{S}^{5})}\leq R_{0},$$
there exists a control $h\in L^{2}([0,T];H^{2}(\mathbb{S}^{5}))$ with $\|h\|\leq C$ such that the unique solution $u\in X^{2,b}_{T}$ satisfies $u(T)=u_{1}$.
\end{theorem}
Under the standard transformation $w = e^{-it}u$, the equation becomes an equation for $w$; for simplicity, we rename $w$ back to $u$. With the control $h = a(x)(1-\Delta_{g})^{-2}\big(a(x)u_{t}\big)$, the stabilization system reads
\begin{equation}\label{damped}
	\begin{cases}
		i\partial_{t}u+\Delta^{2}_{g}u-\beta\Delta_{g}u+|u|^{2k}u+u=a(x)(1-\Delta_{g})^{-2}\big(a(x)u_{t}\big),& (t,x)\in [0,T]\times M,\\ u(0,x)=u_{0}\in H^{2}(M).
	\end{cases}
\end{equation}
By  the direct computation, one can observe that the damped equation \eqref{damped} does not conserve its energy:
\begin{equation*}
	E(u)(T)-E(u)(0)=-\int_{0}^{T}\|(1-\Delta_{g})^{-1}\big(a(x)\partial_{t}u\big)\|_{L^{2}(M)}^{2}\,\dd t.
\end{equation*}The presence of damping term helps us obtain exponential decay of the energy. To state our stabilization theorem, we first introduce the Unique Continuation Principle assumption (UCP).
\begin{assumption}[Unique Continuation Principle (UCP)] \label{UCP-1}
  Let $\omega$ be a nonempty open subset of $M$ and $\beta\in\mathbb{R}^{+}$. For every $T>0$ and
  every $b_{1}, b_{2}\in C^{\infty}([0,T]\times M)$, the only solution
  $u\in C^{\infty}([0,T]\times M)$ of
  \[
  \begin{cases}
    i\partial_{t}u+\Delta^{2}_{g}u-\beta\Delta_{g}u
       +b_{1}(t,x)u+b_{2}(t,x)\overline{u}=0, & (t,x)\in[0,T]\times M,\\[2pt]
    u=0, & (t,x)\in[0,T]\times\omega,
  \end{cases}
  \]
  is the trivial one $u\equiv0$.
\end{assumption}
Now, we state our results.
\begin{theorem}[Stabilization]\label{thm2}
Suppose that $1\leq d\leq4$, and that $M=\mathbb{S}^{5}$ when $d=5$.	Let the open subset $\omega$ satisfy Assumption \ref{GCC}, and assume that the Assumption \ref{UCP-1} holds. Let $a\in C^{\infty}(\omega)$. Then for every $R_{0}>0$, there exist $C(R_{0})$ and $\gamma(R_{0})$ such that 
	\begin{equation}
		\|u(t)\|_{H^{2}(M)}\leq Ce^{-\gamma t}\|u_{0}\|_{H^{2}(M)}, ~t>0,
	\end{equation}
	holds for every solution $u$ of \eqref{damped} with initial data $u_{0}\in H^{2}(M)$ obeying $\|u_{0}\|_{H^{2}(M)}\leq R_{0}$.
\end{theorem}

Let us briefly review the study of controllability and stabilization for the dispersive equation.  
The geometric control condition (GCC)  introduced by Bardos-Lebeau-Rauch \cite{Bardos} has served as a sufficient condition to yield the exact controllability for  the linear dispersive equation in recent years.
However, (GCC) is not necessary in certain geometries.  On the  three-dimensional flat torus $\mathbb{T}^{3}$,  Jaffard \cite{Jaffard} proved that any open subset can be the control domain. We refer to the  monograph \cite{Komornik} and \cite{Tucsnack} for more results on periodic domains. On the other hand, the unique continuation property (UCP) usually serves as a crucial step in stabilization. However, (UCP) frequently appears as an assumption on general Riemannian manifolds.
For the classical Schr\"odinger equation on spheres, Assumption \ref{UCP-1} can be proved via Carleman estimates, as shown in Laurent's appendix \cite{Laurent}.
However, extending such arguments to the general compact manifold   is substantially more delicate. 
Instead of using Carleman estimates,   Loyola \cite{Loyola2025} recently  found that the unique continuation results for Schr\"odinger operators with partially analytic coefficients can help us  deduce the unique continuation for the linearized equation. The key ingredient used in his paper is the abstract unique continuation proved by Robbiano–Zuily \cite{Robbiano-Zuily}, in which they require that the Hamiltonian satisfy a certain transversality condition. For the second-order case, it can be verified directly. For the higher-order case, this type condition usually fails, as pointed out in Filippas-Laurent-L\'eautaud \cite{Laurent-Leautaud}. Indeed, the principle symbol of fourth-order Schr\"odinger operator is $p_4(t,x,\xi)=|\xi|^4_g$. Writing $p_{m,\phi}(x,\xi)=p_{m}(x,\xi+i\tau \dd\phi(x))$, we observe that 
\begin{equation}
    \begin{aligned}
        &p_{4,\phi}(x,\xi)=p_{2,\phi}^{2},\\ &\{p_{4,\phi},\phi\}=2p_{2,\phi}\{p_{2,\phi},\phi\},\\& \{\overline{p_{4,\phi}},p_{4,\phi}\}=2\{\overline{p_{4,\phi}},p_{2,\phi}\}p_{2,\phi}+2\{\overline{p_{2,\phi}},p_{4,\phi}\}\overline{p_{2,\phi}}.
    \end{aligned}
\end{equation}
From these identities, $p_{4,\phi}=0$ is equivalent to $p_{2,\phi}=0$, which implies $\{p_{4}(x,\xi),\phi\}=\{\overline{p_{4,\phi}},p_{4,\phi}\}=0$ on $\{(x,\xi)|p_{2,\phi}(x,\xi)=0\}$. Thus the pseudoconvexity condition
\begin{equation}
    p_{4,\phi}=\{p_{4,\phi},\phi\}=0,\,\,\xi_{t}=0,\,\,\tau>0\Longrightarrow \frac{1}{i}\{\overline{p_{4,\phi}},p_{4,\phi}\}>0
\end{equation}
is never satisfied. Thus, through the example of Filippas-Lauarent-L\'eautaud, we see that Robbiano–Zuily's result is applicable to the second-order Schrödinger operator but fails for the fourth-order one, which justifies the necessity of Assumption \ref{UCP-1}.
\subsection{Outline of the proof}
We adapt the strategy of \cite{Dehman} to establish the global controllability and stabilization.  We now outline the proofs of the main theorems stated in the previous subsection.

The first step is to establish the Strichartz estimates for the fourth-order Schr\"odinger operator, which will be a key ingredient in the proof of the local well-posedness. Using the Kato-Rellich theorem, Helffer-Sj\"ostrand formula, and the analytical arguments in \cite{Burq-AJM}, we derive the Strichartz estimates for fourth-order Schr\"odinger operator $L=i\partial_t+\Delta^2_g-\beta\Delta_g$ on manifolds with loss of regularity. From these estimates, we further deduce the well-posedness of the system in the energy space $H^{2}(M)$ for $d\leq4$. For $d=5$ (with $M=\mathbb{S}^{5}$), we follow an analogous strategy from \cite{Burq-Multilinear} to prove multilinear Strichartz estimates in the Bourgain space, thereby establishing the well-posedness for equations \eqref{equ-1} and \eqref{damped}. Using the microlocal defect measure and commutator estimates, we can derive the propagation of singularities following the ideas from \cite{Dehman}.

The second step is to establish the $L^2$ observablility estimate for the fourth-order Schr\"odinger operator $i\partial_{t}+\Delta^{2}_{g}-\beta\Delta_{g}$. Notably, this inequality is equivalent to the linear controllability of the system via the Hilbert Uniqueness Method (HUM). To verify the observability inequality, we use techniques from semiclassical analysis: specifically, semiclassical microlocal defect measures and the Wigner distribution.

The last step is  to prove Theorem \ref{thm1} and Theorem \ref{thm2}, i.e., the semi-global  controllability and stabilization for nonlinear Schr\"odinger equations on $M$. It is well-known that if we obtain the stabilization result and local exact controllability, we can extend the controllability to the semi-global result. Let us briefly sketch the proof of local exact controllability and stabilization.  Following the classical argument, we analyze the control problem for the corresponding linear equation:
\begin{equation*}
i\partial_{t}u+(\Delta_{g}^{2}-\beta\Delta_{g})u=h.
\end{equation*}
Applying the \textit{Hilbert Uniqueness Method} (see Section 5 for details), constructing the  control operator $\Lambda: u_0\mapsto h$ is equivalent to proving the observability inequality
\begin{equation}\label{obs}
\|u_{0}\|_{L^2}^{2}\leq C\int_{0}^{T}\int_{\omega}|e^{it(\Delta_g^2-\beta\Delta_g)}u_{0}|^2\,\dd x\dd t.
\end{equation}

After a precise analysis of the linear problem, we follow the ideas in \cite{Dehman,Laurent} and decompose the solution $u$ into $u=v+\Psi$, where $v$ and $\Psi$ solves the equations
	\begin{equation}
			\begin{cases}
				i\partial_{t}v+\Delta^2_g v-\beta\Delta_g v+|u|^{2k}u=0,\\ v(T)=0
			\end{cases}
		\end{equation}
		and
		\begin{equation}
			\begin{cases}
				i\partial_{t}\Psi+\Delta^{2}_{g}\Psi-\beta\Delta_{g}\Psi=A\Phi,\\ \Psi(T)=0.
			\end{cases}
		\end{equation}
        where $A\Phi=\varphi(1-\Delta_g)^{-2}(\varphi\Phi)$.
The goal is then to choose a suitable initial datum $\Psi_0$, after which the system is completely determined. Let us define $$\mathcal{N}: \Psi_0\mapsto u_0-v|_{t=0}.$$ Thus, seeking the control function $h$ is reduced to finding a fixed point for $\mathcal{N}$, provided that $\|u_0\|_{H^s}$ is small enough (see Section \ref{sec:Loccon} for details). 

Having established the local exact controllability, our objective is to establish the stabilization result.  By direct calculation, we obtain the energy dissipation
\begin{align}
E(u(t))-E(u_0)=-\int_0^t\big\|(1-\Delta_g)^{-1}(a(x)\partial_tu)\big\|_{L^2(M)}\,\dd t.
\end{align} Establishing the stabilization result is reduced to showing that the following energy decay inequality $$E(0)\leq C\int_{0}^{T}\|(1-\Delta_{g})^{-1}(a(x)\partial_{t}u)\|^{2}_{L^{2}}\,\dd t.$$  The proof relies on a contradiction argument and applications of propagation estimates.

\subsection{Structure of the article}
The paper is organized as follows. In Section \ref{sec:Pre}, we collect some basic properties of harmonic analysis tools and function  spaces on $M$. We also prove some Strichartz estimates in various settings and show the local well-posedness for NLS with a forced term and damping term. In Section \ref{sec:Procompreg}, we prove the propagation of regularity and compactness, which will be crucial in proving the controllability. In Section \ref{sec:Obe}, we prove the observability inequality using semiclassical analysis. In Section \ref{sec:Loccon} and \ref{sec:Sta}, we prove the exact controllability  and stabilization respectively.



\section{Preliminaries}\label{sec:Pre}
		
In this section, we collect several harmonic analysis tools involving the Strichartz estimates, function spaces and prove the local well-posedness for certain Schr\"odinger system with forcing and damping terms.

 We first define the Lebesgue norm on compact manifolds. Denote by $g$ the canonical metric of $M$ and $dx$ the volume form associated with the metric $g$. Then the $L^p$ norm can be expressed as 
        \begin{align*}
          \|f\|_{L^p(M)}=\Big(\int_M|f(x)|^p dx\Big)^\frac1p.  
        \end{align*}
        The Sobolev space $H^s(M)$ and Bourgain space $X^{s,b}$
 are both $L^2$-based and their norms can be defined by using the spectral resolution. 		Let $L^{2}(M;\mathbb{C})$ be a Hermitian inner product space. Now, we introduce the properties of eigenvalues and eigenfunctions of $-\Delta_g$. Let $\lambda_k$ be the $k$-th eigenvalue of $-\Delta_g$ counted with multiplicity.  Denote by $e_k$ the $k$-th eigenfunction of $-\Delta_g$. The family $\{e_{k}\}_{k=1}^{\infty}$ form an orthonormal basis of $L^2(M)$. For a given $f\in L^2(M)$, we have
        $$f=\sum_{k=1}^\infty \pi_kf=\sum_{k=1}^\infty\langle f,e_k\rangle e_k(x),\qquad-\Delta_g e_k(x)=\lambda_ke_k(x),$$
        where $\pi_k$ is
  the spectral projector onto $k$-th eigenspace. We then define the Sobolev space $H^s(M)$ as follows,
\begin{equation*}
H^s(M)=\Big\{u\in \mathcal{S}^\prime(M):\|u\|_{H^{s}(M)}^{2}=\sum_{k\in\mathbb{N}}\langle\lambda_{k}\rangle^{s}\|\pi_{k}u\|_{L^{2}(M)}^{2}<+\infty\Big\}		
		\end{equation*}
 where  $\langle\cdot\rangle=(1+|\cdot|^{2})^{\frac{1}{2}}$. By the spectral decomposition, the linear fourth-order Schr\"odinger flow can be rewritten as
 \begin{align*}
        e^{it(\Delta_g^2-\beta\Delta_g)}f=\sum_{k=1}^\infty e^{it(\lambda_k^2+\beta\lambda_k)}\pi_kf,\,\,\forall f\in L^2(M).
        \end{align*}
   Then, we can define the Bourgain space as follows.
         \begin{definition}[$X^{s,b}$ space]Let $s\geq0$ and $b\in\R$, then we denote
         \begin{align*}
             X^{s,b}(\R\times M)=\left\{u\in\mathcal{S}^\prime(\R,L^2(M)):\|u\|_{X^{s,b}(\R\times M)}<\infty\right\},
         \end{align*}
           where
$$\|u\|^{2}_{X^{s,b}}=\sum_{k\in\mathbb{N}}\langle\lambda_{k}\rangle^{s}\|\langle\tau+\lambda_{k}^{2}+\beta\lambda_{k}\rangle^{b}\widehat{\pi_{k}u}(\tau)\|_{L^{2}(\mathbb{R}_{\tau}\times M)}^{2}=\|e^{-it(\Delta_{g}^{2}-\beta\Delta_{g})}u(t)\|^{2}_{H^{b}(\mathbb{R};H^{s}(M))}.$$
Here $\widehat{\pi_{k}u}(\tau)$ denotes the time Fourier transform of $\pi_{k}u$, $\beta\in\mathbb{Q}^{+}$.
		For convenience, we set $u^{\Diamond}=e^{-it(\Delta_{g}^{2}-\beta\Delta_{g})}u(t)$. Moreover, if $b=\infty$, then $X^{s,\infty}=\bigcap\limits_{b<\infty}X^{s,b}$. For $s\leq0$, the Bourgain space can be defined by duality.
         \end{definition}
   For bounded $T>0$, we can define the restricted Bourgain space equipped with the norm 
\begin{equation*}
\|u\|_{X_{T}^{s,b}}=\inf\big\{\|\tilde{u}\|_{X^{s,b}}:\tilde{u}=u~\textit{on}~(0,T)\times M\big\}.
\end{equation*}
		
In the following lemma, we collect some properties of the $X_T^{s,b}$. The proof is rather standard, and we refer to Tao's textbook \cite{Tao} for the detailed proofs. 
		
\begin{lemma}[Basic properties of $X^{s.b}$, \cite{Tao}]\,
        \begin{enumerate}
			\item For $s\geq0$ and $b>\frac12$, we have the Sobolev embedding $X_{T}^{s,b}([0,T]\times M)\hookrightarrow C([0,T],H^s(M))$.
			\item If $s_{1}\leq s_{2}$, $b_1\leq b_2$, then the embedding $X^{s_2,b_2}\subset X^{s_1,b_1}$ is continuous.
			\item For every $s_{1}<s_{2}$, $b_1<b_2$, and $T>0$, we have $X^{s_2,b_2}_{T}\Subset X^{s_1,b_1}_{T}$.
			\item For $0<\theta<1$, the complex interpolation theorem gives the interpolation space $$(X^{s_1,b_1},X^{s_2,b_2})_{[\theta]}=X^{(1-\theta)s_{1}+\theta s_{2},(1-\theta)b_{1}+\theta b_{2}}.$$
		\end{enumerate}
        \end{lemma}
        Next, we provide two multiplier estimates in $X^{s,b}$.
		\begin{lemma}
			Let $\varphi\in C_{0}^{\infty}(\mathbb{R})$ and $u\in X^{s,b}$, then $\varphi(t)u\in X^{s,b}$. If $u\in X_{T}^{s,b}$, then we have $\varphi(t)u\in X_{T}^{s,b}$.
\end{lemma}

\begin{lemma}[Boundedness for pseudo-differential operator, \cite{Robertoamo}]
			Let $-1\leq b\leq 1$ and let $B$ be a pseudo-differential operator in the space variable of order $\rho$. For any $u\in X^{s,b}$ we have $Bu\in X^{s-\rho-3|b|,b}$. Similarly, $B$ maps $X^{s,b}_{T}$ into $X^{s-\rho-3|b|,b}_{T}$.
		\end{lemma}
Next, we list several elementary estimates and the proof can be found in \cite{Laurent}.
\begin{lemma}\label{nonhomogeneous term}
			Let $b$ and $b'$ satisfy $0<b'<\frac{1}{2}<b$ and $b+b'\leq1$. For $f\in H^{-b'}(\mathbb{R})$ and inhomogeneous term $F(t)=\Psi(\frac{t}{T})\int_{0}^{t}f(t')\,\dd t'$ with $\Psi\in C_{0}^{\infty}([-1,1])$, we have for $T\leq1$
			$$\|F\|_{H^{b}(\mathbb{R})}\leq CT^{1-b-b'}\|f\|_{H^{-b'}(\mathbb{R})}.$$
		\end{lemma}
		
\begin{lemma}
Let $0<b<1$ and $u\in X^{s,b}$ then the function 
\begin{equation}
f:(0,T]\to\mathbb{R},\,\,   t\mapsto \|u\|_{X_{t}^{s,b}} 
\end{equation}
is continuous. Moreover, if $b>\frac{1}{2}$, there exists $C_{b}$ such that $\lim\limits_{t\to0}f(t)\leq C_{b}\|u(0)\|_{H^{s}}$.
\end{lemma}
      
\begin{lemma}\label{covering}
Let $0<b<1$. If $\cup_{k}(a_{k},b_{k})$ is a finite covering of $[0,1]$, then there exists a constant $C$ depending only on the covering such that for every $u\in X^{s,b}$, 
\begin{equation*}
\|u\|_{X^{s,b}_{[0,1]}}\leq C\sum_{k}\|u\|_{X^{s,b}_{[a_{k},b_{k}]}}.
\end{equation*}
\end{lemma}  

As a direct consequence of Lemma \ref{nonhomogeneous term}, one has the following inhomogenuous estimate in Bourgain space:
        \begin{lemma}\label{lem-inho-X}
Let $b,b^\prime,F(t)$ be defined as in Lemma \ref{nonhomogeneous term}. Then we have            $$\big\|F(t)\big\|_{X_T^{s,b}}\leq CT^{1-b-b^\prime}\|f\|_{X_T^{s,-b^\prime}}.$$
        \end{lemma}
        \subsection{Linear and bilinear Strichartz estimate on $M$}
        In this subsection, we prove the linear Strichartz estimate for fourth-order Schr\"odinger operator $\Delta_g^2+V$ under the standard Kato-Rellich condition with $d\geq1$. Also, we prove the bilinear Strichartz estimate on spheres $\Bbb S^5$.  

        We start with the linear Strichartz estimate.
\begin{proposition}\label{prop:semi-Strichartz}
			For all $2\leq p,q\leq\infty$ satisfying
			\begin{align}\label{admissible}
			(p,q,d)\neq(2,\infty,2),\,\,\frac{2}{p}+\frac{d}{q}\leq\frac{d}{2}.
			\end{align}
			Then for any $u_0\in H^{\gamma_{p,q}+\frac{3}{p}}(M)$, the following Strichartz estimate holds
			\begin{align*}
				\big\|e^{it(-\Delta_g^2+V)}u_0\big\|_{L_{t}^pL_x^q}\leq C(I,p,q)\|u_0\|_{H^{\gamma_{p,q}+\frac{3}{p}}(M)},
			\end{align*}
			where $\gamma_{p,q}=\frac d2-\frac{d}{q}-\frac{4}{p}$ and $V$ satisfies the Kato-Rellich condition $\|Vf\|_{L^{2}(M)}\leq C\|f\|_{H^{2}(M)}$.
		\end{proposition}
		
		\begin{proof}
			First, we prove Strichartz estimate for Schr\"odinger operator $A=\Delta_{g}^{2}+V$ on $(\R^d,g)$. Using the Littlewood-Paley theory, it is sufficient to prove that
			\begin{align}\label{unitary decomposition}
				\big\|\psi(hD)u\|_{L_t^pL_x^q}\leq C\|\varphi(h^4A)u_0\|_{H^{\frac{3}{p}+\gamma_{p,q}}}+Ch^{2-\frac{1}{p}+\gamma_{p,q}}\|u_0\|_{H^{\frac{3}{p}+\gamma_{p,q}}}.
			\end{align}
			 We observe that
			\begin{align*}
				\psi(hD)u=&\psi(hD)\varphi(h^4A)u+\psi(hD)\big[\varphi(h^4\Delta_g^2)-\varphi(h^4A)\big]u\\&+\psi(hD)\big[1-\varphi(h^4\Delta_g^2)\big]u\\
				=&:I_1+I_2+I_3.
			\end{align*}
			Next, we will estimate these terms separately.

			By the Duhamel formula, we can write the solution as
			\begin{align*}
				u(t)=e^{it\Delta_g^2}u_0-\frac{1}{i}\int_{0}^te^{i(t-s)\Delta_g^2}(Vu(s))ds
				.\end{align*}
			Recall that the following semi-classical Strichartz estimates hold in the short time interval $J=[-\varepsilon h^3,\varepsilon h^3]$, which is proved in \cite{Dinh},
			\begin{align*}
				\big\|e^{it\Delta_g^2}\chi(x,hD)u_0\big\|_{L_t^pL_x^q(J\times\R^d)}\lesssim h^{-\gamma_{p,q}}\|u_{0}\|_{L^2(\R^d)}.
			\end{align*}
			Next, we recall the almost orthogonal lemma from \cite{Burq-AJM}, which will be used in treating the terms $I_{1}$ and $I_3$.
			\begin{lemma}[Orthogonal estimate]\label{lem-Ortho}
				For a fixed $\psi\in C_0^\infty(\R^d\setminus\{0\})$, then there exists $\varphi\in C_0^\infty(\R\setminus\{0\})$ such that for all $h\in(0,1]$, $\sigma>0$, $N>0$, $f\in C_0^\infty(\R^d)$, and $P$ is a $m$-th order pseudodifferential operator, we have
				\begin{align*}
					\big\|\psi(hD)(1-\varphi(h^mP))f\big\|_{H^\sigma(\mathbb{R}^d)}\leq C_{\sigma,N}h^N\|f\|_{L^2(\mathbb{R}^d)}.
				\end{align*}
			\end{lemma}
			
			We apply Lemma \ref{lem-Ortho} with $m=4$ and $P=(-\Delta_g)^2$ to obtain
			\begin{align}\label{cutoff}
\big\|\psi(hD)u\big\|_{L_t^pL_x^q(J\times\R^{d})}&\lesssim h^{-\gamma_{p,q}}\big(\|u_0\|_{L^2}+\|\varphi(h^4\Delta_g^2)(Vu)\|_{L_t^1L_x^2}\big)+C_Nh^N\|Vu\|_{L_t^1L_x^2}.
			\end{align}
			Since $\varphi$ is a compactly supported smooth function, the direct calculation and Kato-Rellich condition yield that
			\begin{align*}
				\big\|\varphi(h^4\Delta_g^2)(Vf)\big\|_{L^2(\R^d)}\leq Ch^{-2}\|\varphi(h^4\Delta_g^2)(Vf)\|_{\dot H^{-2}}\lesssim h^{-2}\|f\|_{L^2}.
			\end{align*}
		    Taking the time integral on $J$ and by \eqref{cutoff}
			\begin{align*}
				\big\|\psi(hD)u\big\|_{L_t^pL_x^q(J\times\R^d)}\lesssim h^{-\gamma_{p,q}}\|u_0\|_{L^2}+h^N\|u_0\|_{H^2}.
			\end{align*}
			Replacing $u_0$ by $\phi(h^4A)u_0$ and using the fact that
			\begin{align*}
				\big\|\varphi(h^4A)u_0\big\|_{H^2}\leq \frac{C}{h^2}\big\|\varphi(h^4A)u_0\big\|_{L^2},
			\end{align*}
			we get
			\begin{align*}
				\big\|\psi(hD)\varphi(h^4A)u\big\|_{L^p_tL_x^q(J\times\R^d)}\lesssim h^{-\gamma_{p,q}}\big\|\varphi(h^4A)u_0\big\|_{L^2}.
			\end{align*}
			Summing over all $J$ with length $|J|\leq \varepsilon h^{3}$, we have the following estimate in time interval $I$
			\begin{align*}
				\big\|\psi(hD)\varphi(h^4A)u\big\|_{L_t^pL_x^q(I\times\R^2)}\lesssim h^{-\gamma_{p,q}-\frac{3}{p}}\big\|\varphi(h^4A)u_0\big\|_{L^2}\lesssim \big\|\varphi(h^4A)u_0\big\|_{H^{\gamma_{p,q}+\frac3p}}.
			\end{align*}
            This completes the estimate of the term $I_{1}$. Using the orthogonal estimate, the term $I_3$ is similar to $I_1$. Therefore, it remains to estimate $I_2$. First, we claim that for $\varphi\in C_0^\infty(\R)$ and $0\leq s,r\leq1$, it holds
			\begin{align}\label{diff}
				\big\|\varphi(h^4A)-\varphi(h^4\Delta_g^2)\big\|_{H^r\to H^s}\leq Ch^{2-s+r}.
			\end{align}
			Indeed, it follows from the Helffer-Sj\"ostrand formula \cite[Chapter 2]{Davies},
			\begin{align*}
				\varphi(h^4A)-\varphi(h^4\Delta^2_g)&=-\frac{1}{\pi}\int_{\mathbb{C}}\overline{\partial}\tilde{\varphi}(z)\big((z-P_1)^{-1}-(z-P_2)^{-1}\big)\,L(\dd z)\\
				&=-\frac{1}{\pi}\int_{\mathbb{C}}\overline{\partial}\tilde{\varphi}(z)(z-P_1)^{-1}h^4V(z-P_2)^{-1}\,L(\dd z),
			\end{align*}
			where $L(\dd z)=\dd x\dd y$ is the Lebesgue measure in $\mathbb{C}$, $\tilde{\varphi}$ is the almost analytical extension of $\varphi$ and $P_1=h^4A$ and $P_2=h^4\Delta_g^2$. By the spectral theory, we have the $L^2$ bound
			\begin{align*}
				\big\|(z-P_j)^{-1}\big\|_{L^2\to L^2}\leq \sup_{\lambda\in\R}|z-\lambda|^{-1}=|\Im z|^{-1}.
			\end{align*}
			For the $H^s\to H^s$ boundedness, we replace the resolvent by $(hD)^s(z-P_j)^{-1}$. Typically, for $s\geq r$ we have
			\begin{align*}
				\big\|(z-P_1)^{-1}\big\|_{H^{s}\to H^{r}}&\lesssim C|\Im z|^{-1}h^{s-r},\\
				\big\|(z-P_2)^{-1}\big\|_{H^2\to H^r}&\lesssim C|\Im z|^{-1}h^{2-r}.
			\end{align*}
			Applying \eqref{diff} with $s= \frac{2}{p}\leq \frac d2-\frac{2}{q}$ and $r=\gamma_{p,q}+\frac{3}{p}$, one has
			\begin{align*}
				\|I_2\|_{L_t^p(I,L^q(M))}\lesssim h^{2+\frac{1}{p}+\gamma_{p,q}}\|u_0\|_{H^{\frac{3}{p}+\gamma_{p,q}}(M)}.
			\end{align*}
			 Combining the analysis above, we finish the proof of \eqref{unitary decomposition}.
			 Finally, by partition of unity theorem on manifold and the previous Strichartz estimates on $(\mathbb{R}^{d},g)$, we complete the proof of Proposition \ref{prop:semi-Strichartz}.
\end{proof}

 \begin{remark}
 In this article, we apply this proposition with $V=-\beta\Delta_{g}$ where $\beta\in\R$ is fixed.
 \end{remark}
		
As a consequence, we have the following inhomogeneous Strichartz estimate.
\begin{lemma}[Inhomogeneous Strichartz estimate] Let $(p,q)\in\R^2$ obey \eqref{admissible}, then we have
\begin{align*}
\left\|\int_{0}^{t}e^{i(t-s)(\Delta_g^2-\beta\Delta_g)}F(u(s))\,ds\right\|_{L^{p}_{t}L^{q}_{x}([0,T]\times M)}\leq C\|F\|_{L_t^1([0,T],H^\frac1p(M))}.
\end{align*}
\end{lemma}
		
Now, we establish the $L^2$ bilinear Strichartz estimate. 
\begin{lemma}[Bilinear estimate]\label{bilinear estimate}
Let $I$ be a compact interval of time and $s>\frac{3}{2}$. Then for $L\geq K\geq1$, $\beta\in\mathbb{Q}_{+}$ and $f_1, f_2\in L^{2}(\mathbb{S}^{5})$, we have
\begin{equation}\label{bilinear1}
\|e^{it(\Delta_{g}^{2}-\beta\Delta_{g})}P_{K}f_1 e^{it(\Delta_{g}^{2}-\beta\Delta_{g})}P_{L}f_2\|_{L^{2}_{t,x}(I\times\mathbb{S}^{5})}\lesssim K^{s}\|P_{K}f_1\|_{L^{2}}\|P_{L}f_2\|_{L^{2}},
\end{equation}
where $P_{K}f=\sum\limits_{K\leq k\leq2K}\pi_{k}f$.
\end{lemma}
		
		\begin{proof}  Using the spectral resolution theorem, we have
			\begin{equation*}
				e^{it(\Delta_{g}^{2}-\beta\Delta_{g})}P_{K}f_{1}e^{it(\Delta_{g}^{2}-\beta\Delta_{g})}P_{L}f_{2}=\sum_{\substack{K\leq k<2K\\ L\leq l< 2L}}e^{it(\mu_k^4+\mu_{l}^{4})}\pi_{k}f_{1}\pi_{l}f_{2},
			\end{equation*}
			where $$\mu_{k}^{4}=\big[k(k+4)\big]^{2}+\beta k(k+4)~\textit{and}~\mu_{l}^{4}=\big[l(l+4)\big]^{2}+\beta l(l+4).$$
 Since $\beta\in\mathbb{Q}_{+}$, there exists $p,q\in \mathbb N$ such that $\beta=\frac{p}{q}$ with  $\operatorname{gcd}(p,q)=1$.  Then we can rewrite       
\begin{equation*}
				e^{it(\Delta_{g}^{2}-\beta\Delta_{g})}P_{K}f_{1}e^{it(\Delta_{g}^{2}-\beta\Delta_{g})}P_{L}f_{2}=\sum_{\substack{K\leq k<2K\\ L\leq l< 2L}}e^{it\frac{1}{q}(\tilde\mu_k^4+\tilde\mu_{l}^{4})}\pi_{k}f_{1}\pi_{l}f_{2},
			\end{equation*}
where $\tilde\mu_k^4+\tilde\mu_{l}^{4}=q(k(k+4))^2+pk(k+4)+q(l(l+4))^{2}+pl(l+4)$. Since $\tilde\mu_k^4+\tilde\mu_l^4$ is an integer, the mapping
$$t\mapsto e^{it(\mu_k^4+\mu_l^4)}=e^{i\frac{t}{q}(\tilde\mu_k^4+\tilde\mu_l^4)}$$
is $2\pi q$ periodic in $t$.
Hence, by splitting the time interval into copies of $[0,2\pi q]$,
we may assume $I=[0,2\pi q]$ without loss of generality
(up to a constant factor depending on $|I|$).
Rescaling time as $s=\frac{t}{q}$, the exponential becomes $e^{is(\tilde\mu_k^4+\tilde\mu_l^4)}$,
which is $2\pi$ periodic in $s$.
Thus we can further restrict to $I=[0,2\pi]$ and we rename $s$
back to $t$ for notational simplicity.
			
Using Plancherel theorem (with respect to time variable $t$) and the bilinear spectral cluster estimate established in \cite{Burq-Bilinear},
\begin{equation}
 \big\|\pi_k u\pi_{l}v\big\|_{L^2(\Bbb S^5)}\lesssim \min\{k,l\}^{\frac{3}{2}}\|\pi_{k}u\|_{L^{2}(\Bbb S^5)}\|\pi_l v\|_{L^2(\Bbb S^{5})},
\end{equation}
we obtain
\begin{align}
&\|e^{it(\Delta_{g}^{2}-\beta\Delta_{g})}P_{K}f_{1}e^{it(\Delta_{g}^{2}-\beta\Delta_{g})}P_{L}f_{2}\|^{2}_{L_{t,x}^{2}(I\times\mathbb{S}^{5})}\nonumber\\\leq& \sum_{\tau=0}^{\infty}\Big\|\sum_{(k,l)\in A_{K,L}(\tau)}\pi_{k} f_{1}\pi_{l} f_{2}\Big\|^{2}_{L^{2}(\mathbb{S}^{5})}\nonumber\\ \leq& \sup_{\tau\in\mathbb{N}}\# A_{K,L}(\tau)K^{3}\|P_{K}f_{1}\|^{2}_{L^{2}(M)}\|P_{L}f_{2}\|^{2}_{L^{2}(M)},\label{bilinear estimates}
\end{align}
where $A_{K,L}(\tau):=\{(k,l)\in\mathbb{N}^{2}:\tau=\tilde\mu_{k}^{4}+\tilde\mu_{l}^{4}, K\leq k<2K, L\leq l<2L\}$. Next, it remains to count the number of lattice points contained in $A_{K,L}(\tau)$. To achieve this, we need a further reduction.  Notice that  the time frequency parameter $\tau$ can be written as
\begin{align*}
\frac{\tau}{q}=&(k(k+4))^{2}+\frac{p}{q}k(k+4)+(l(l+4))^{2}+\frac{p}{q}l(l+4)\\
=&\Big(k(k+4)+\frac{p}{2q}\Big)^{2}+\Big(l(l+4)+\frac{p}{2q}\Big)^{2}-\frac{p^{2}}{2q^{2}}.
\end{align*}
Then, it is equivalent to count the number of lattice sets satisfying the following condition
\begin{align*}
4q\tau+2p^{2}=(2qk(k+4)+p)^{2}+(2ql(l+4)+p)^{2}.
\end{align*}

According to \cite[Lemma 3.2]{Burq-Bilinear}, for every $\varepsilon>0$, there exists $C>0$ such that for every positive integers $\tau$ and $K$,
\begin{equation*}
\#\{(k_1,k_2)\in\mathbb{N}^{2}: K\leq k_{1}<2K, k_1^2+k_2^2=\tau\}\leq C_{\varepsilon}K^{\varepsilon}.
\end{equation*}

Then,  we have $\#A_{K,L}(\tau)\leq C_{\varepsilon} K^{\varepsilon}$. This together with  \eqref{bilinear estimates} yields the desired estimate \eqref{bilinear1}.
\end{proof}

As a consequence of Lemma \ref{bilinear estimate}, transfer principle \cite[Lemma 2.3]{Burq-Bilinear}, and the  standard arguments as in \cite{Burq-Bilinear}, we can obtain the following trilinear estimates.
\begin{lemma}\label{nonlinear estimates1}
On $\mathbb{S}^{5}$, for every $r\geq s>\frac{3}{2}$, there exist $0<b'<\frac{1}{2}$ and $C>0$ such that for any $u$ and $\tilde{u}$ in $X^{r,b'}$, 
\begin{align}
\big\||u|^{2}u\big\|_{X^{r,-b'}}\leq& C\|u\|^{2}_{X^{s,b'}}\|u\|_{X^{r,b'}},\\
\big\||u|^{2}\tilde{u}\big\|_{X^{r,-b'}}\leq& C\|u\|_{X^{s,b'}}\|u\|_{X^{r,b'}}\|\tilde{u}\|_{X^{r,b'}},\\
\big\||u|^{2}u-|\tilde{u}|^{2}u\big\|_{X^{s,-b'}}\leq& C\big(\|u\|_{X^{s,b'}}^{2}+\|\tilde{u}\|_{X^{s,b'}}^{2}\big)\|u-\tilde{u}\|_{X^{s,b'}}.
\end{align}
\end{lemma}

\subsection{Well-posedness for the fourth-order Schr\"odinger system}
In this part, we study the well-posedness for the following Cauchy problem 
\begin{align}\label{force-NLS}
			\begin{cases}
				i\partial_tu+(\Delta_g^2-\beta\Delta_g)u=-|u|^{\alpha-1}u+h(x,t),\\
				u(0,x)=u_0(x).
			\end{cases}
\end{align}

\begin{lemma}[Local well-posedness]
Let $s>\frac{d}{2}-\frac{1}{\alpha-1}$ with $\alpha\geq3$, then the solution to \eqref{force-NLS} is locally well-posed in $H^s(M)$ if $h(t,x)\in L_{\rm loc}^1(\R,H^s(M))$. In particular, $u\in C([0,T],H^s(M))\cap L^p([0,T],L^\infty(M))$ for some $T>0$ and $p>\alpha-1$. Furthermore, if $d<5$ and $s\geq2$, the solution is globally in time.
\end{lemma}
		
\begin{proof}
We proceed by using the classical Picard iteration. Assume that $T>0$ and choosing $p>\alpha-1$ such that $s>\frac{d}{2}-\frac{1}{p}$. Then we define the target space as
\begin{equation}\label{equ:ETdef}
E_T:=C([-T,T],H^s(M))\cap L^p([-T,T],W^{\beta_0,q}(M)),
\end{equation}
where $q$ satisfies $\frac{2}{p}+\frac{d}{q}\leq\frac{d}{2}$ and $\beta_0=s-\frac{1}{p}>\frac{d}{q}$. Then we have the Sobolev embedding $W^{\beta_0,q}\hookrightarrow L^\infty$. Next, we define the norm of space $E_T$ via
\begin{equation}\label{Equation of ET}
\|u(t)\|_{E_T}:=\sup_{t\in[-T,T]}\|u(t)\|_{H^s(M)}+\big\|(I-\Delta_g)^\frac{\beta_0}{2}u\big\|_{L^p([-T,T],L^q(M))}.
\end{equation}
For convenience, we  only consider the positive time direction. Another direction can be obtained in a same way.
			
Using the Duhamel formula, the solution map can be represented by
\begin{align*}
\mathcal{T}(u)(t)=e^{it(\Delta_g^2-\beta\Delta_{g})}u_0-i\int_{0}^te^{i(t-s)(\Delta_g^2-\beta\Delta_{g})}[h(s)-|u(s)|^{\alpha-1}u(s)]\,ds.
\end{align*}
Applying Strichartz estimate and fractional chain rule, we have
\begin{align*}
\big\|\mathcal{T}(u)\big\|_{E_T}&\lesssim\|u_0\|_{H^s(M)}+\int_0^T\|h(s)-|u|^{\alpha-1}u(s)\|_{H^s(M)}\,ds\\
&\lesssim\|u_0\|_{H^s(M)}+\|h\|_{L^1([0,T],H^s(M))}+\int_0^T\|u\|_{L^\infty(M)}^{\alpha-1}\|u(s)\|_{H^s(M)}\,ds\\
&\lesssim\|u_0\|_{H^s(M)}+\|h\|_{L^1([0,T],H^s(M))}+T^{1-\frac{\alpha-1}{p}}\|u\|_{E_T}^{\alpha-1}\|u\|_{L^\infty H^s([0,T]\times M)}\\
&\lesssim\|u_0\|_{H^s(M)}+\|h\|_{L^1([0,T],H^s(M))}+T^{1-\frac{\alpha-1}{p}}\|u\|_{E_T}^{\alpha-1}\|u\|_{E_T}.
\end{align*}
Repeating the procedure, we have
\begin{align*}
\big\|\mathcal{T}(u)-\mathcal{T}(v)\big\|_{E_T}\lesssim T^{1-\frac{\alpha-1}{p}}(1+\|u\|_{E_T}+\|v\|_{E_T})\|u-v\|_{E_T}.
\end{align*}
Taking $T$ such that $T^{1-\frac{\alpha-1}{p}}\ll1$,  the continuity method and Picard iteration imply that $u$ is a local and unique solution to \eqref{force-NLS} on $[0,T]$. 
			
For $1\leq d\leq4$ and  $s\geq2$, in order to extend the solution to globally in time, we consider the energy functional $E(t)$. By direct computation, we have
\begin{align*}
E(T)-E(0)\leq \int_0^T\int_{M}|\Delta_gu||\Delta_gh|+|\nabla_gu||\nabla_gh|\,dx\,dt+\int_0^T\int_{M}|u|^\alpha|h|\,dx\,dt.
\end{align*}
Since the first two terms is bounded, we only need to control the potential energy part.  Using H\"older's inequality, we obtain
\begin{align*}
\int_0^T\int_{M}|u|^\alpha|h|\,dx\,dt&\lesssim\int_0^T\|h(t)\|_{L^{p_2}(M)}\big\||u|^{\alpha-1}u\|_{L^{p_1}(M)}\,dt\\
&\lesssim\int_0^T\|u\|_{L^{\alpha+1}}^{\alpha+1}\|h(t)\|_{H^2(M)}\,dt\\
&\lesssim C\int_0^TE(t)\|h(t)\|_{H^2(M)}\,dt,
\end{align*}
where $p_1=\frac{\alpha+1}{\alpha}$ and $p_2=\alpha+1$. Then by Gronwall's inequality, we claim that $E(T)$ has the uniform upper bound. Hence, the solution exists globally in time. The uniqueness can be obtained by the similar argument as above.
\end{proof}
		
Next, we prove the well-posedness for the fourth-order Schr\"odinger equation with a damping term
\begin{align}\label{damping-NLS}
\begin{cases}
i\partial_tu+(\Delta_g^2-\beta\Delta_g)u=-|u|^{\alpha-1}u-u+a(x)(I-\Delta_g)^{-2}a(x)\partial_tu,&(t,x)\in\R\times M,\\
u(0,x)=u_0(x).
\end{cases}
\end{align}
		
\begin{lemma}
Let $1\leq d\leq4$, $u_0\in H^2(M)$ and $a(x)\in C^\infty(M)$ be a non-negative real-valued function. Then, there exists an unique solution $u\in C([0,\infty),H^2(M))$ to \eqref{damping-NLS}.
\end{lemma}
		
\begin{proof}
Since the proof is similar to the above lemma,  we only give some sketch here. Denote by $Jv$  the damping operator with  $Jv:=(I-ia(x)(I-\Delta_g)^{-2}a(x))v$. By using the boundedness of  elliptic pseudo-diferential operator, $J$ is bounded on $H^s(M)$. Therefore, we can rewrite  equation \eqref{damping-NLS} to
\begin{align}\label{DNLS-2}
\begin{cases}
\partial_tv-i(\Delta_g^2-\beta\Delta_g)v-R_0v-i|u|^{\alpha-1}u=0,\\
v(0,x)=v_0(x):=Ju_0\in H^2(M),
\end{cases}
\end{align}
where $v=Ju$ and  $R_0=-i(\Delta_g^2-\beta\Delta_g){(ia(x)(1-\Delta_{g})^{-2}a(x))J^{-1}}-iJ^{-1}$ is a zero-order pseudo-differential operator. By the Duhamel principle, the solution to \eqref{DNLS-2} can be rewritten as
\begin{align*}
\mathcal{T}(v)(t)=e^{it(\Delta_g^2-\beta\Delta_g)}v_0+\int_0^te^{i(t-s)(\Delta_g^2-\beta\Delta_g)}\big[R_0v+i|u|^{2k}u\big](s)\,ds.
\end{align*}
Using the similar strategy as in the above lemma, we can show that $\mathcal T$ is a bounded and surjective map in $E_T$ which leads to the local well-posedness. Also, by differentiating the time derivative of $E(t)$ yields the global well-posedness.
\end{proof}

For $d=5$, due to the loss of derivatives appearing in the linear Strichartz estimates, we cannot establish the wellposedness for the cubic damping equation in $E_{T}$. To overcome this, we will establish the analogue result in the $X^{s,b}$ space.
        
\begin{lemma}
Let $u_{0}\in H^{s}(\mathbb{S}^{5})$, $s\geq2$ and $a(x)\in C^{\infty}(M)$ be a non-negative real-valued function. Suppose that $h\in L_{\rm loc}^{2}(\mathbb{R},H^{s}(M))$. Then, for any $T>0$, there exists an unique solution on $[0,T]$ in $X^{s,b}_{T}$ with $b>\frac12$ to the Cauchy problem 
 \begin{equation*}
                \begin{cases}
                i\partial_{t}u+\Delta_{g}^{2}u-\beta\Delta_{g}u+|u|^{2}u=h+a(x)(1-\Delta_{g})^{-2}(a(x)u_{t}), &(t,x)\in [0,T]\times\mathbb{S}^{5},\\ u(0)=u_{0}.
                \end{cases}
            \end{equation*}
\end{lemma}

\begin{proof}
Denoting the damping operator $Jv:=\big(1-ia(x)(1-\Delta_{g})^{-2}a(x)\big)v$, we know that $J$ is an isomorphism of $X^{s,b}$. Let $v=Ju$, we rewrite system as before:
\begin{equation}
\begin{cases}
\partial_{t}v-i(\Delta^{2}_{g}-\beta\Delta_{g})v-R_{0}v-i|u|^{2}u=-ih, &(t,x)\in[0,T]\times\mathbb{S}^{5} \\ v=Ju,\\ 
v(0)=v_{0}=Ju_{0}\in H^{s},
\end{cases}
\end{equation}
where $R_{0}=-i(\Delta_{g}^{2}-\beta\Delta_{g})(ia(x)(1-\Delta_{g})^{-2}a(x))J^{-1}-iJ^{-1}$. 
			
Firstly, we note that if $h\in L^{2}([0,T],H^{s})$, it holds $h\in X^{s,-b'}_{T}$ for $b'\geq0$. By Duhamel's formula, the solution map can be written as 
\begin{equation}\label{duhamel complex}
\mathcal{T}(v)(t)=e^{it(\Delta_{g}^{2}-\beta\Delta_{g})}v_{0}+\int_{0}^{t}e^{i(t-\tau)(\Delta_{g}^{2}-\beta\Delta_{g})}\big[R_{0}v+i|u|^{2}u-ih\big]\,\dd\tau.
\end{equation}
In the sequel, we will prove that $\mathcal{T}(v)$ is a contraction mapping in $X^{s,b}_{T}$. Let $\psi\in C_{0}^{\infty}(\mathbb{R})$ be equal to $1$ on interval $[-1,1]$. From the  definition of $X^{s,b}$, we obtain
\begin{equation*}
\big\|\psi(t)e^{it(\Delta_{g}^{2}-\beta\Delta_{g})}v_{0}\big\|_{X^{s,b}}=\|\psi\|_{H^{b}}\|v_{0}\|_{H^{s}}.
\end{equation*}
Hence, for $T\leq1$ we have
\begin{equation}
\big\|e^{it(\Delta^{2}_{g}-\beta\Delta_{g})}v_{0}\big\|_{X^{s,b}_{T}}\leq C\|v_{0}\|_{H^{s}}\leq C\|u_{0}\|_{H^{s}}.
\end{equation}
In addition, Lemma \ref{lem-inho-X} and nonlinear estimate 
\begin{align}\label{key-est-X}
&\Bigg\|\int_{0}^{t}e^{i(t-\tau)(\Delta_{g}^{2}-\beta\Delta_{g})}[R_{0}v+i|u|^{2}u-ih](\tau)\,\dd\tau\Bigg\|_{X^{s,b}_{T}}\\ \leq& CT^{1-b-b'}\big\|R_{0}v+i|u|^{2}u-ih\big\|_{X^{s,-b'}_{T}}\\ \leq&CT^{1-b-b'}\Big(\|R_{0}v\|_{X^{s,-b'}_{T}}+\big\||u|^{2}u\big\|_{X^{s,-b'}_{T}}+\|h\|_{X^{s,-b'}_{T}}\Big)\\ \leq& CT^{1-b-b'}\|v\|_{X_{T}^{s,b}}\big(1+\|v\|^{2}_{X_{T}^{2,b}}\big)+CT^{1-b-b'}\|h\|_{X^{s,-b'}_{T}}.
\end{align}
We then conclude that 
\begin{equation}\label{nonlinear estimate for T}
\|\mathcal{T}(v)\|_{X^{s,b}_{T}}\leq C\|v_0\|_{H^s}+CT^{1-b-b'}\|v\|_{X_{T}^{s,b}}(1+\|v\|^{2}_{X_{T}^{2,b}})+CT^{1-b-b'}\|h\|_{X_{T}^{s,-b'}}
\end{equation}
and similarly
\begin{align}\label{contraction}
\|\mathcal{T}(v)-\mathcal{T}(\tilde{v})\|_{X^{s,b}_{T}}\leq CT^{1-b-b'}\|v-\tilde{v}\|_{X^{s,b}_{T}}\Big(1+\|v\|^{2}_{X^{s,b}_{T}}+\|\tilde{v}\|^{2}_{X^{s,b}_{T}}\Big).
\end{align}
Taking $T$ small such that $T^{1-b-b'}\leq\frac{1}{2}$, $\mathcal{T}$ is contracting on a ball of $X_{T}^{s,b}$. Combining  \eqref{contraction} and Duhamel equation \eqref{duhamel complex}, we get the uniqueness of solution in the class $X_{T}^{s,b}$.
			
Next, we prove the propagation of regularity. If $u_{0}\in H^{s}$ with $s>2$, we claim that the solution exists in $X^{2,b}_{T}$ at time $T$ and another time $\tilde{T}$ for the existence in $ X^{s,b}_{\tilde{T}}$. By uniqueness in $X^{2,b}_{T}$, the two solutions are  same on $[0,\min\{\tilde{T},T\}]$. Without loss of generality, we can assume that $\tilde{T}<T$. Then $\|u(t,\cdot)\|_{H^{s}}$ tends to $\infty$ as $t\to\tilde{T}$ while $\|u(t,\cdot)\|_{H^{2}}$ remains bounded. Using the  local well-posedness result  in $H^{2}$ and Lemma \ref{covering}, we get that $\|u\|_{X^{2,b}_{\tilde{T}}}$ is bounded. Utilizing  the nonliner estimate \eqref{nonlinear estimate  for T} on a small interval  $[\tilde{T}-\varepsilon,\tilde{T}]$, with $\varepsilon$ small enough such that $C\epsilon^{1-b-b'}(1+\|v\|_{X^{2,b}_{\tilde{T}}}^{2})<\frac{1}{2}$, we get 
\begin{equation*}
\|v\|_{X^{s,b}_{\tilde{T}}}\leq C\|u(\tilde{T}-\varepsilon)\|_{H^{s}}+\|h\|_{X_{T}^{s,-b'}}.
\end{equation*}
Therefore, $u\in X^{s,b}_{\tilde{T}}$, and this is a  contradiction since  $\|u(t,\cdot)\|_{H^{s}}\to\infty$ as $t\to\tilde{T}$.
			
Next, we use the classical energy method to get global existence. The energy  functional of  $u$ can be formulated as
\begin{equation}\label{nonlinear energy}
E(t)=\frac{1}{2}\int_{M}|\Delta_{g}u|^{2}\,\dd x+\frac{\beta}{2}\int_{M}|\nabla_{g} u|_{g}^{2}\,\dd x+\frac{1}{4}\int_{M}|u|^{4}\,\dd x+\frac{1}{2}\int_{M}|u|^2\,\dd x.
\end{equation}
If $h=0$ and $a=0$, the energy functional is conserved. Now, we only consider the case that the force term $h$ and damping term $a$ is non-trivial. Multiplying the equation by $\partial_{t}\bar{u}$, we can get 
\begin{align*}
&E(t)-E(0)\\=&-\int_{0}^{t}\|(1-\Delta_{g})^{-1}a(x)\partial_{t}u\|^{2}_{L^{2}}\,\dd t-\operatorname{Re}\int_{0}^{t}\int_{M}h\overline{\partial_{t}u}\,\dd x\,\dd t\\=&-\int_{0}^{t}\|(1-\Delta_{g})^{-1}a(x)\partial_{t}u\|^{2}_{L^{2}}\,\dd t-\operatorname{Re}\int_{0}^{t}\int_{M}(J^{-1})^{*}h\overline{\partial_{t}v}\,\dd x\,\dd t\\=&-\int_{0}^{t}\|(1-\Delta_{g})^{-1}a(x)\partial_{t}u\|^{2}_{L^{2}}\,\dd t\\&-\operatorname{Re}\int_{0}^{t}\int_{M}(J^{-1})^{*}h\Big(\overline{i(\Delta_{g}^{2}-\Delta_{g})v+R_{0}v+i|u|^{2}u-ih}\Big)\\ \leq & C\int_{0}^{t}\|\Delta_{g}(J^{-1})^{*}h\|_{L^2}\|\Delta_{g}v\|_{L^{2}}\,\dd t+C\int_{0}^{t}\|\nabla_{g}(J^{-1})^{*}h\|_{L^2}\|\nabla_{g}v\|_{L^{2}}\,\dd t \\&+C\int_{0}^{t}\|h\|_{L^{2}}\|v\|_{L^{2}}+C\int_{0}^{t}\|h\|_{L^{4}}\|u\|_{L^{4}}^{3}\,\dd t+C\int_{0}^{t}\|h\|^{2}_{L^{2}}\,\dd t\\ \leq& C\int_{0}^{t}\|h\|_{H^{2}}^{2}\Big(\|u\|^{2}_{L^{2}}+\|\nabla u\|^{2}_{L^{2}}+\|\Delta_{g}u\|^{2}_{L^{2}}\Big)+\int_{0}^{t}(\|h\|_{L^{4}}^{4}+\|u\|_{L^{4}}^{4})\\&+\|h\|^{2}_{L^2([0,T]\times M)}
\end{align*}
Consequently, $E(t)$ is bounded on any $[0,T]$ via Gronwall's inequality. Thus, we deduce the global existence of $u$ in $X^{s,b}_{T}$ for every $T>0$. 
\end{proof}

\begin{proposition}\label{linearization}
Assume that $\{u_n\} \subset X_T^{2,b}$ is a bounded sequence of solutions to
\[(3.1) \begin{cases}
i\partial_t u_n + \Delta^{2}_{g} u_n-\beta\Delta_{g}u_{n} + u_n + |u_n|^2 u_n &= a(x)(1 - \Delta)^{-2}a(x)\partial_t u_n, \quad \text{on } [0,T] \times\mathbb{S}^{5}, \\
u_n(0) = u_{n,0} \in H^2(\mathbb{S}^{5})
\end{cases}
\]
such that
$$u_{n,0} \rightharpoonup 0 \quad \text{in } H^2(\mathbb{S}^{5}).$$
Then
$$|u_n|^2 u_n \longrightarrow 0 \quad \text{in } X_T^{2,-b'}.$$
\end{proposition}

\begin{proof}
We aim to show that any subsequence (still denoted by \( u_n \)) possesses a further subsequence converging to 0. Our proof relies on  a  nonlinear estimates in  \( X_T^{s,b} \) space. For  \( \frac{3}{2} < s < 2 \), by Lemma \ref{nonlinear estimates1} we have
\begin{equation} \label{linearnnn}
\| |u_n|^2 u_n \|_{X_T^{2,-b'}} \leq C \| u_n \|_{X_T^{s,b}}^2 \| u_n \|_{X_T^{2,b}},
\end{equation}
where $C$ is independent of $n$. Moreover, by repeating the procedure of \eqref{key-est-X} and replacing $X_T^{2,b}$ by $X_T^{2,b+\varepsilon}$,  we claim that  $u_{n}\in X_T^{2,b+\varepsilon}$ for some $\varepsilon$. By the Sobolev embedding, there exists a subsequence (still denoted $u_{n}$) that  converges weakly in $X^{2,b}_{T}$ and strongly in $X^{s,b}_{T}$ to $u$ in $X^{s,b}_{T}$ with $u(0)=0$. Additionally, $u_{n}(0)$ strongly converges to $0$ in $H^{s}$. By the continuity of solutions in $H^{s}$, $u_{n}$ strongly converges to $0$ in $X^{s,b}_{T}$. Substituting this into \eqref{linearnnn}, we deduce the desired conclusion.
\end{proof}


		
\section{Propagation of compactness and regularity}\label{sec:Procompreg}
		
This section is devoted to establishing several results on the propagation of singularities in dimensions $d\leq5$, which are fundamental to the subsequent proof of the stabilization property.

\subsection{Propgation of compactness}
First, we establish the propagation of compactness in energy space for $1\leq d\leq4$.
		
\subsubsection{Propagation of compactness in energy space: $d\leq4$}
Let $T>0$ and $\{u_n\}_{n\in\mathbb{N}}\subset C([0,T],H^2(M))$ be a sequence of functions such that
\begin{gather}
\sup_{t\in[0,T]}\|u_n(t)\|_{H^2(M)}\leq C,\,\,\sup_{t\in[0,T]}\|u_n(t)\|_{L^2(M)}\to0,\label{con-1}\\
\int_0^T\big\|Lu_n(t)\big\|_{H^2(M)}^2\,dt\to0\,\,\mbox{as}~ n \to\infty,\label{con-2}
\end{gather}
where $L=i\partial_t+(\Delta_g^2-\beta\Delta_g)$ is the fourth-order Schr\"odinger operator. Then, we state the propagation of compactness in the following proposition.

\begin{proposition}[Propagation of compactness]\label{propagation-compactness}  
Let $\{u_n\}$ be the sequences such that $\eqref{con-1}$ and \eqref{con-2} hold. Assume that $u_n\to0$ in $L^2((0,T),H^2(\omega))$ with $\omega\subset M$ an open set satisfying Assumption \ref{GCC}. Then passing through a subsequence, we have $u_n\to0$ in $L^\infty((0,T),H^2(M))$.
\end{proposition}

\begin{proof}
Let $v_n:=(1-\Delta_g)u_n$ be a regularized sequence of $u_n$. Then \eqref{con-1} and \eqref{con-2} become
\begin{gather}
\sup_{t\in[0,T]}\|v_n(t)\|_{L^2(M)}\leq C,\,\,\sup_{t\in[0,T]}\|v_n(t)\|_{H^{-2}(M)}\to0,\label{con-3}\\
\int_0^T\big\|Lv_n(t)\big\|_{L^2(M)}^2\,dt\to0\,\,\mbox{as}~ n \to\infty.\label{con-4}
\end{gather}
Similarly, $u_n\to0$ in $L^2((0,T),H^2(\omega))$ turns to 
$$v_n\to0 \mbox{ in } L^2((0,T),L^2(\omega)).$$ 
Then, it suffices to show that passing a subsequence, there holds
\begin{align*}
v_n\to0 \mbox{ in } L^\infty((0,T),L^2(M)).
\end{align*}
The proof will be divided into three steps. 
			
\noindent\textbf{Step 1. Construction of the semiclassical defect measure.} Under the assumptions \eqref{con-3} and \eqref{con-4}, inspired by \cite{Dehman}, one can show that  there exists a Radon measure $\mu(t,x,\xi)$ supported on $T^*M$ such that
\begin{align}
\lim_{n\to\infty}  (Q(t,x,D_x)v_n,v_n)_{L^2(I\times M)}=\int_{[0,T]\times T^*M}q(t,x,\xi)\,d\mu,
\label{defect-measure}
\end{align}
where $Q$ is a pseudo-differential operator of zero-order and $q(t,x,\xi)$ is the principal symbol of $Q$.
			
\noindent\textbf{Step 2. Propagation along the geodesic flow}. Let $\varphi(t)\in C_0^\infty([0,T])$ and $Q_1(x,D_x)$ be a pseudo-differential operator of order $-3$ with principle symbol $q_{-3}$. We denote by
\begin{align*}
Q(t,x,D_x)=\varphi(t)Q_1(x,D_x).
\end{align*}
For $\varepsilon>0$, we define the regularization of $Q(t,x,D_x)$
\begin{align*}
Q_\varepsilon(t,x,D_x):=Q(t,x,D_x)e^{\varepsilon\Delta_g}=\varphi(t)Q_1^{\varepsilon}(x,D_x).
\end{align*}
Throughout the proof, we denote $(\cdot,\cdot):=(\cdot,\cdot)_{L_{t,x}^2}$ for simplicity. Next, we define the difference term $\alpha_{n,\varepsilon}$ 
\begin{align*}
				\alpha_{n,\varepsilon}&=(Lv_n,Q_\varepsilon^*v_n)_{L^2}-(Q_\varepsilon v_n,Lv_n)_{L^2}\\
				&=([Q_\varepsilon,\Delta_g^2-\beta\Delta_g]v_n,v_n)_{L^2}-i(\partial_t(Q_\varepsilon)v_n,v_n)_{L^2}.
\end{align*}
By the direct calculation, \eqref{con-3} and \eqref{con-4}, we have
\begin{align*}
				|(Lv_n,Q_\varepsilon^*v_n)_{L_{t,x}^2}|&\lesssim\|Lv_n\|_{L_{t,x}^2}\|Q_\varepsilon^*v_n\|_{L_{t,x}^2}\\
				&\lesssim\|Lv_n\|_{L_{t,x}^2}\|v_n\|_{L_t^2H^{-2}}\to0
\end{align*}
as $n\to\infty$.  Similarly, one can obtain
\begin{align*}
|(Q_\varepsilon v_n,Lv_n)_{L^2}|\to0.
\end{align*}
Putting two convergence results together, we have $\sup\limits_{\varepsilon>0}\alpha_{n,\varepsilon}\to0$ as $n\to\infty$.
			
On the other hand,
\begin{align*}
\big|(\partial_t Q_\varepsilon\cdot v_n,v_n)_{L^2}\big|&\leq \|\partial_t Q_\varepsilon\cdot v_n\|_{L_t^2H^{2}}\|  v_n\|_{L_t^2H^{-2}}\to0.
\end{align*}
Then taking $\varepsilon\to0$ and $n\to\infty$ one by one, we get 
\begin{align*}
([Q_\varepsilon,\Delta_g^2-\beta\Delta_g]v_n,v_n)_{L^2}\to0.
\end{align*}
Equivalently, using the microlocal defect measure as in Step 1, it implies
\begin{align*}
\int_{[0,T]\times S^*M}\varphi(t)\big\{q_{-3},|\xi|_x^4+\beta|\xi|_x^2\big\}\,d\mu(t,x,\xi)=0.
\end{align*}
Using the symbol relationship,
\begin{align*}
\big\{q_{-3},|\xi|^4_x+\beta|\xi|_x^2\big\}=(2+\beta)\big\{q_{-3},|\xi|_x^2\big\}
\end{align*}
and the flow property of Hamiltonian $p(\xi)=|\xi|^2$, the microlocal defect measure is invariant under the geodesic flow. More precisely, we have
\begin{align*}
\Phi_{s}(\mu)=\mu,
\end{align*}
where $\Phi_{s}$ denotes the geodesic flow on the cosphere bundle $S^*M=\{(x,\xi)\in T^*M:|\xi|_x^2=1\big\}$.
			
\noindent\textbf{Step 3. Convergence in $\mathbf{L^\infty L^2}$.} In this final step, we prove the convergence of $v_n$ which will complete the proof. By the construction of microlocal measure, we have
\begin{align*}
\lim_{n\to\infty}    (fv_n,v_n)_{L^2((0,T)\times M)}=\int_{(0,T)\times S^*M}f(x)\,d\mu.
\end{align*}
Then, $v_n\to0$ in $L^2((0,T),L^2(M))$ is equivalent to $\mu=0$ on $(0,T)\times S^*M$. From the assumption in this proposition, we have the convergence in $[0,T]\times\omega$ where $\omega $ obeys the GCC condition. Using the fact that $\mu$ is invariant under the geodesic flow,  we have $\mu=0$ on $(0,T)\times S^*M$. Hence, we complete the proof of Proposition \ref{propagation-compactness}.
\end{proof}

Using the linear propagation of compactness, we can show the analogues of nonlinear equation for $1\leq d\leq 4$ and $\alpha\geq3$.

\begin{lemma}
Let $\{u_n\}\in C([0,T],H^2(M))$ be solutions to
			\begin{align*}
				i\partial_tu_n+(\Delta_g^2-\beta\Delta_g)u_n=-|u_n|^{\alpha-1}u_n,
			\end{align*}
			and $u_n(0)$ weakly converges to zero in $H^2(M)$. Then, $|u_n|^{\alpha-1}u_n\to0$ in $L^2((0,T),H^2(M))$,
\end{lemma}

\begin{proof}
Using the fractional chain rule and H\"older's inequality, one has
\begin{align*}
\big\||u_n|^{\alpha-1}u_n\big\|_{L^2([0,T],H^2(M))}^2&\lesssim\int_0^T\|u_n(t)\|_{L^\infty}^{2\alpha-2}\|u_n(t)\|_{H^2}^2\,dt.
\end{align*}
Note that the $H^2$ norm of $u_n$ is bounded uniformly and by energy conservation law,
\begin{align*}
\big\||u_n|^{\alpha-1}u_n\big\|_{L^2([0,T],H^2(M))}^2&\lesssim\int_0^T\|u_n(t)\|_{L^\infty}^{2(\alpha-1)}\,dt.
\end{align*}
Taking $(p,q)\in [2,\infty)^{2}$ such that $\frac{2}{p}\leq d(\frac{1}{2}-\frac{1}{q})$ and $p>2\alpha-2$, we have $W^{\sigma,q}\hookrightarrow L^\infty$ if $\sigma>\frac dq$. Then by Gagliardo-Nirenberg, we have
\begin{align*}
				\|u_n\|_{L^\infty}\leq \|u_n\|_{L^q}^{1-\theta}\|u_n\|_{W^{\sigma_1,q}}^\theta,
			\end{align*}
			where $\sigma_1=2-\frac1p>\frac dq$. Using H\"older's inequality, we further have
			\begin{align*}
				\int_0^T\|u_n(t)\|_{L^\infty}^{2(\alpha-1)}\,dt\leq\Big(\int_0^T\|u_n\|_{L^q}^r\,dt\Big)^\eta\Big(\int_0^T\|u_n\|_{W^{\sigma_1,q}}^p\,dt\Big)^\frac{2\theta(\alpha-1)}{p}
\end{align*}
for some $\eta,r>0$. By the nonlinear estimate in studying the local well-posedness, we know that the second term is uniformly bounded. Since $u_n(0)$ converges to $0$ weakly in $H^2$ and the fact that $H^2\hookrightarrow L^2$ is compact, we claim that $u_n(0)\to 0$ in $L^2$. Then the first term goes to zero. Hence, the nonlinear term converges to zero in $L^2H^2$ topology.
\end{proof}

\subsubsection{Propagation of compactness in Bourgain space: $d=5$}
\begin{proposition}\label{bourgain space propgation}
Let $s\in\mathbb{R}$ and $u_{n}$ be a sequence of solutions to 
\begin{equation}
i\partial_{t}u_{n}+\Delta_{g}^{2}u_{n}-\beta\Delta_{g}u_{n}=f_{n} 
\end{equation}
satisfying the uniform bound $\|u_{n}\|_{X^{s,b}_{T}}\leq C$ and the convergence properties 
$\|u_{n}\|_{X_{T}^{s-3+3b,-b}}\to0$,  $\|f_{n}\|_{X_{T}^{s-3+3b,-b}}\to0$ holds for some $b\in[0,1]$. 
			
Then, there exist a subsequence $\{u_{n'}\}$ and a positive defect measure $\mu$ on $(0,T)\times S^{*}M$ so that for any tangential pseudo-differential operator $A=A(t,x,D_{x})$ of order $2s$ with principal symbol $\sigma(A)=a_{2s}(t,x,\xi)$, we have the convergence:
\begin{equation*}(A(t,x,D_{x})u_{n'},u_{n'})_{L^{2}((0,T)\times M)}\to\int_{(0,T)\times S^{*}M}a_{2s}(t,x,\xi)\,\dd\mu(t,x,\xi), ~\textit{as}~ n'\to\infty.
\end{equation*}
Furthermore, the measure $\mu$ is invariant under  the geodesic flow $G_{t}$ on $S^{*}M$ for all $t\in\mathbb{R}$, i.e.
\begin{equation*}
G_{t}(\mu)=\mu.
\end{equation*}
\end{proposition}

\begin{proof}
Define the operator $L=i\partial_{t}+\Delta_{g}^{2}-\beta\Delta_{g}$. Let $\varphi(t)\in C_{0}^{\infty}((0,T))$ and  $B(x,D_{x})$ be a pseudo-differential operator of order $2s-3$, with principal symbol $b_{2s-3}$. Denote by $A(t,x,D_{x})=\varphi(t)B(x,D_{x})$. For any $\varepsilon>0$, we denote the smoothing operator by $A_{\epsilon}=\varphi B_{\epsilon}=Ae^{\varepsilon\Delta}$.
			
Since $A_{\varepsilon}u_{n}$ and $A_{\varepsilon}^{*}u_{n}$ are $C^{\infty}$ functions, we have 
\begin{equation*}
(Lu_{n},A^{*}_{\varepsilon}u_{n})_{L^{2}((0,T)\times M)}=(f_{n},A_{\varepsilon}^{*}u_{n})_{L^{2}((0,T)\times M)}
\end{equation*}
and 
\begin{equation*}
(A_{\varepsilon}u_{n},Lu_{n})_{L^{2}((0,T)\times M)}=(A_{\varepsilon}u_{n},f_{n})_{L^{2}((0,T)\times M)}.
\end{equation*}
			
We then denote
\begin{align*}
\alpha_{n,\varepsilon}=&(Lu_{n},A^{*}_{\varepsilon}u_{n})_{L^{2}((0,T)\times M)}-(A_{\varepsilon}u_{n},Lu_{n})_{L^{2}((0,T)\times M)}\\=&\big([A_{\varepsilon},L] u_{n},u_{n}\big)_{L^{2}((0,T)\times M)}\\=& ([A_{\varepsilon},\Delta^{2}_{g}-\beta\Delta_{g}]u_{n},u_{n})_{L^{2}((0,T)\times M)}-i(\partial_{t}(A_{\epsilon})u_{n},u_{n})_{_{L^{2}((0,T)\times M)}}.
\end{align*}
Notice that $\partial_{t}A_{\epsilon}$ is of order $2s-3$ which is uniform in $\varepsilon$, then 
\begin{align*}
\sup_{\varepsilon}\big| i(\partial_{t}(A_{\epsilon})u_{n},u_{n})_{_{L^{2}((0,T)\times M)}}\big|\leq& C\|\partial_{t}(A_{\varepsilon})u_{n}\|_{X^{-s+3-3b,b}_{T}}\|u_{n}\|_{X^{s-3+3b,-b}_{T}}\\ \leq& C\|u_{n}\|_{X^{s,b}_{T}}\|u_{n}\|_{X^{s-3+3b,-b}_{T}}
\to0, \end{align*}
as $n\to\infty$.
			
In addition, we also have
\begin{align*}
\alpha_{n,\varepsilon}=&(f_{n},A^{*}_{\varepsilon}u_{n})-(A_{\epsilon}u_{n},f_{n})\\
|(f_{n},A_{\varepsilon}^{*}u_{n})_{L^{2}((0,T)\times M)}|\leq&\|f_{n}\|_{X^{s-3+3b,-b}_{T}}\|A^{*}_{\varepsilon}u_{n}\|_{X_{T}^{-s+3-3b,b}}\\ \leq&\|f_{n}\|_{X^{s-3+3b,-b}_{T}}\|u_{n}\|_{X^{s,b}_{T}}.
\end{align*}
Thus, $\sup_{\varepsilon}|(f_{n},A_{\varepsilon}^{*}u_{n})_{L^{2}((0,T)\times M)}|\to0$ as $n\to\infty$. Similarly, we have  $$\sup_{\varepsilon}|(A_{\epsilon}u_{n},f_{n})_{L^{2}((0,T)\times M)}|\to0,$$ which deduce that $\sup_{\varepsilon}\alpha_{n,\varepsilon}\to0$ as $n\to \infty$. Finally, taking the supreme norm over $\varepsilon\leq \varepsilon_0$, we get 
\begin{align*}
&\lim_{\varepsilon_{0}\to0}\sup_{\varepsilon\leq\varepsilon_{0}}([A_{\varepsilon},\Delta^{2}_{g}-\beta\Delta_{g}]u_{n},u_{n})_{L^{2}((0,T)\times M)}\\=&(\varphi[B,\Delta^{2}_{g}-\Delta_{g}]u_{n},u_{n})_{L^{2}((0,T)\times M)}\to 0, \textit{as}~ n\to\infty.
\end{align*}
This implies that 
\begin{equation*}
\int_{(0,T)\times S^{*}M}\varphi(t)\{|\xi|_{x}^{4}+\beta|\xi|^{2}_{x},b_{2s-3}\}\,\dd\mu(t,x,\xi)=0.
\end{equation*}
Combining with the fact that $\{|\xi|_{x}^{4}+\beta|\xi|^{2}_{x},b_{2s-3}\}=(2+\beta)\{|\xi|^{2}_{x},b_{2s-3}\}$ on $S^{*}M$, we get the propagation along geodesic flows. 
\end{proof}
		
\begin{corollary}
Let $s\in\mathbb{R}$ and assume that $\omega\subset M$ satisfies GCC condition. Assume that the nonegtive function $a\in C^{\infty}(\omega)$. Let $u_{n}$ be a bounded sequence in $X_{T}^{s,b'}$ with $0<b'<\frac{1}{2}$, which weakly convergences to $0$ and satisfying 
\begin{equation}
\begin{cases}
i\partial_{t}u_{n}+(\Delta_{g}^{2}-\beta\Delta_{g})u_{n}:=f_{n}\to0, &\textit{in}~ X^{s,-b'}_{T},\\ a(x)u_{n}\to0,&\textit{in}~L^{2}([0,T],H^{s}).
\end{cases}
\end{equation}
Then, we have $u_{n}\to0$ in $X_{T}^{s,1-b'}$ as $n\to\infty$.
\end{corollary}
	
\begin{proof}
Let $(u_{n_{k}})$ be a subsequence of $(u_{n})$. Since $b^\prime\in(0,\frac12)$ and the manifold is compact,  we can apply Proposition \ref{bourgain space propgation} directly. Then, we can get a micro-local defect measure associated to $(u_{n_{k}})$ in $L^{2}([0,T],H^{s})$ which propagates along  geodesic flows with infinite propagation speed. On the other hand, $a(x)u_{n}\to0$ provides  that $a(x)\mu=0$. Due to the fact  $i\partial_{t}u_{n}+(\Delta_{g}^{2}-\beta\Delta_{g})u_{n}\to0, $ in~ $X^{s,-b'}_{T}$ and the ellipticity of $a(x)$ on $\omega$, we deduce $\mu=0$ on $(0,T)\times S^{*}M$. As a direct consequence, we have  $(u_{n_{k}})\to 0$ in $L^{2}([0,T],H^{s})$ and $u_{n}\to 0$.
				
Then, we can choose $t_{0}$ such that $u_{n_{k}}(t_{0})\to0$ in $H^{s}$. Using Lemma \ref{lem-inho-X}, it holds
\begin{equation}
\big\|\int_{t_{0}}^{t}e^{i(t-\tau)(\Delta_{g}^{2}-\beta\Delta_{g})}f_{n}(\tau)\,\dd\tau\big\|_{X_{T}^{s,1-b'}}\leq C\|f_{n}\|_{X^{s,-b'}_{T}}
\end{equation} 
for $T\leq1$. Then it follows from the Duhamel formula,  $$u_{n}\to0 \mbox{ in } X_{T}^{s,1-b'}.$$ To extend $T$ larger, we can use the Poisson summation in $t$ and  then remove the restriction on $T\leq1$.
				
\end{proof}

\subsection{Propagation of regularity }In this part, our main goal is to prove the propagation of regularity. 
\begin{proposition}[Propagation estimate]\label{nonlinear propagation}
Suppose that  $T>0$. Assume that $u\in C([0,T], H^{2}(M))$ is a solution of 
\begin{equation}\label{propagation of regularity}
i\partial_{t}u+\Delta_{g}^{2}u-\beta\Delta_{g}u=h,
\end{equation} 
where $h\in L^2([0,T], H^2(M))$. In addition, let $\omega$ satisfy the Geometric Control Condition (GCC) and if $u\in L^2([0,T],H^{2+\rho}(\omega))$ for some $\rho\leq\frac32$, then   $u\in L^2((0,T),H^{2+\rho}(M))$. In particular, if $u\in C^\infty((0,T),\omega)$, then $u\in C^\infty ((0,T),M).$
\end{proposition}
Indeed, this is the direct application of the following lemma. 		

\begin{lemma}\label{propagation-regularity}
Let $T>0$ and $s\in\mathbb{R}$. Assume that $u\in C([0,T],H^s(M))$ is a solution to  \eqref{propagation of regularity}
with $h\in L^2([0,T],H^s(M))$. For some $\rho_0\in T^{*}_{x_{0}}M$, we assume that there exists a $0$-order pseudo-differential operator $\psi(x,D_{x})$, elliptic at $\rho_{0}$, satisfying $\psi(x,D_{x})u\in L^{2}([0,T],H^{s+\nu}(M))$ for some $\nu\leq\frac{3}{2}$. Then for each  bicharacteristic flow $\Phi_{\rho_{0}}(t)$ originating at $\rho_{0}$, there exists a $0$-order pseudo-differential operator $\eta(x,D_{x})$ elliptic at $\rho_{1}\in \Phi_{\rho_{0}}(t)$, such that $\eta(x,D_{x})u\in L^{2}([0,T];H^{s+\nu}(M))$. 
\end{lemma}
		
\begin{proof}
The proof is divided into two steps.
			
\textbf{Step 1}. A commutator type estimates. Throughout the proof, we abbreviate $I=(0,T)$. Consider $B(x,D_x)$ the pseudo-differential operator on $M$ of order $2r-3$ where $r=s+\nu$. We observe that $A(t,x,D_x):=\varphi(t)B$ is a pseudo-differential operator with same order as $B$. For the smooth sequence $\{u_n\}$ with $u_n=(1-\frac{1}{n}\Delta_g)^{-2}u$, let $h_n=(1-\frac{1}{n}\Delta_g)^{-2}h$, then $u_n$ satisfies the following equation
\begin{align*}
i\partial_tu_n+(\Delta_g^2-\beta\Delta_g)u_n=h_n.
\end{align*}
Then, the direct computation implies
\begin{align*}
&(h_n,A^{*}u_n)_{L^2(I\times M)}-(Au_n,h_n)_{L^2(I\times M)}\\
&=([A,\Delta_g^2-\beta\Delta_g]u_n,u_n)_{L^2(I\times M)}-i(\varphi^\prime(t)Bu_n,u_n)_{L^2(I\times M)}.
\end{align*}
For the left-hand side of the identity, it has uniformly bound w.r.t. $n$. Indeed, notice that the order of $A$ is $2r-3$, then 
\begin{align}
\big|(h_n,Au_n)_{L^2}\big|&=\big|\big((1-\Delta_g)^{\frac{1}{2}(r-\frac{3}{2})}h_n,(1-\Delta_g)^{\frac{1}{2}(-r+\frac{3}{2})}Au_n\big)_{L^2}\big|\nonumber\\
&\lesssim\|h_n\|_{L^1(I,H^s(M))}\|u_n\|_{L^\infty(I,H^s(M))}.\label{estimate of h}
\end{align}
Here we used the fact that $(1-\Delta_{g})^{\frac{1}{2}(-r+\frac{3}{2})}A$ is of order $-r+\frac{3}{2}+2r-3=r-\frac{3}{2}=s+\nu-\frac{3}{2}\leq s$. So, it maps $H^{s}(M)$ into $L^{2}(M)$, while for the estimate of $\|(1-\Delta_{g})^{\frac{1}{2}(r-\frac{3}{2})}h_{n}\|_{L^{1}(I,L^{2}(M))}\leq C\|h_{n}\|_{L^{2}(I,H^{s}(M))}$, we use the fact that $s-(r-\frac{3}{2})\geq0$. Since $\varphi(t)\in C_{0}^\infty(0,T)$, one can also view $\partial_{t}\varphi B$ as a pseudo-differential operator with same order as $A$. Thus we use the same technique as \eqref{estimate of h}  to obtain 
\begin{equation}
(\varphi'(t)Bu_{n},u_{n})\leq \|u_{n}\|^{2}_{L^{\infty}(H^{s})}
\end{equation}
with the fact $s-(-s+2r-3)=2s-2r+3\geq0$. Thus we get the following estimate
\begin{equation}
\Bigg|\int_{0}^{T}([A,\Delta_{g}^{2}-\beta\Delta_{g}]u_{n},u_{n})\,\dd t\Bigg|\lesssim\|u_{n}\|_{L^{\infty}(H^{s})}^{2}+\|h_{n}\|^{2}_{L^{2}(H^{s})}\lesssim\|u\|_{L^{\infty}(H^{s})}^{2}+\|h\|^{2}_{L^{2}(H^{s})}.
\end{equation}
As a consequence, the commutator term has the uniformly bound which is independent of $n$. 
			
\textbf{Step 2}. Microlocal propagation. First, a direct computation shows that
\begin{align}\label{computation}
\big\{|\xi|_x^4+\beta|\xi|_x^2,b_1\big\}=(2|\xi|_x^2+\beta)\big\{|\xi|_x^2,b_1\big\},
\end{align}
where $b_1=b_1(x,\xi)$ is a symbol and $\{\cdot,\cdot\}$ denotes Poisson bracket. Now fix $\rho_{1}=\Phi_{t}(\rho_{0})$. Our aim is to transport a symbol supported near $\rho_{1}$ along the Hamiltonian flow of $p(x,\xi)=|\xi|^{4}_{g}+\beta|\xi|_{g}^{2}$, up to a remainder localized near $\rho_{0}$ that carries information from $\rho_{0}$. Recall that for a given $\rho_0\in T^*M\setminus\{0\}$ and $\rho_1\in\Phi_{t}(\rho_0)$, $V_{0},V_{1}$ are the conical neighborhoods of $\rho_0$ and $\rho_1$ respectively, then for every symbol $c(x,\xi)$ of order $r$ supported in $V_{0}$, there exists a symbol $b(x,\xi)$ of order $2r-1$  such that 
\begin{align*}
\frac{1}{i}\big\{|\xi|_x^2,b\big\}=|c(x,\xi)|^2+\mathfrak{r}(x,\xi),
\end{align*}
where $\mathfrak{r}(x,\xi)$ is a symbol of order $2r$ supported on $V_{1}$. Substituting  into \eqref{computation}, we find that for every symbol $c_1=c_1(x,\xi)$ of order $r$ supported in $V_{0}$, there exists a symbol $b_0(x,\xi)$ of order $2r-3$ and a remainder term $\mathfrak{r}_1(x,\xi)$ of order $2r$ such that 
\begin{align*}
\frac{1}{i}\big\{|\xi|_x^{4}+\beta|\xi|_x^2,b_0\big\}=|c_1(x,\xi)|^2+\mathfrak{r}_1(x,\xi).
\end{align*}
If we choose $c(x,D_x)$ elliptic at $\omega_1$ then we can conclude that 
        $$\int_0^T\big\|c(x,D_x)u_n\big\|_{L^2}^2\,dt\leq C.$$
            By taking $\psi(x,D_x):=(I-\Delta_g)^{-\frac s2}c(x,D_x)$ implies the required estimate.
		\end{proof}

		Next, we will use the propagation of regularity to deduce the unique continuation property for the nonlinear Schr\"odinger equation. 
		
		\begin{theorem}\label{low dimension}
			Let $1\leq d\leq 4$ and $\omega$ satisfy the Assumption \ref{GCC}. If the solution $u\in C([0,T],H^{2}(M))$ to \eqref{4NLS} satisfies $\partial_{t}u=0$ in $(0,T)\times\omega$, then $\partial_{t}u=0$ in $(0,T)\times M$. Moreover, $u=0$ on $[0,T]\times M$. 
		\end{theorem}
		\begin{proof}
			Let $u\in C([0,T],H^{2}(M))$ be a solution to \eqref{4NLS}, the condition $\partial_tu=0$ on $(0,T)\times\omega$ yields that $$\Delta^{2}_{g}u-\beta\Delta_{g}u+|u|^{2k}u=0,\,\,(t,x)\in (0,T)\times\omega.$$ By using the elliptic regularity and bootstrap argument, we have  $u\in C^{\infty}((0,T)\times\omega)$, which then allows us to apply the propagation of regularity result to conclude that $C^\infty((0,T)\times M)$. Next, taking the time derivatives to \eqref{4NLS} and denote by $v=\partial_tu$, we get
            \begin{align}
            \label{equa-v}\begin{cases}i\partial_tv+(\Delta_g^2-\beta\Delta_g)v+b_1(t,x)v+b_2(t,x)\overline v=0,&(t,x)\in(0,T)\times M,\\
                v=0,&(t,x)\in(0,T)\times\omega,\end{cases}
            \end{align}
            with $b_1,b_2\in C^\infty((0,T)\times M)$.
            Under the Assumption \ref{UCP-1}, we obtain $v=0$ on $(0,T)\times M$. Thus, by multiplying $\overline{u}$ on both side of equation \eqref{4NLS}, we obtain that $$\int_M\big(\big|\Delta_gu\big|^2+\beta\big|\nabla_gu\big|^2+|u|^{2k+2}\big)\,dx=0,$$ which implies that $u\equiv0$.
		\end{proof}
        Using the similar strategy, we can prove the unique continuation property for \eqref{4NLS} in Bourgain space.
		\begin{theorem}\label{five dimension}
			Let $d=5$ and $\omega$ satisfy the Geometric Control Condition Assumption \ref{GCC}.  Let $\frac{1}{2}<b\leq1$ and $u\in X^{2,b}_{T}$ be a solution to 
            \begin{equation}\label{lower term u}
            iu_{t}+(\Delta_g^2-\beta\Delta_g) u =- |u|^2 u -u\end{equation} 
            with $\partial_{t}u=0$ in $(0,T)\times\omega$. Then $u=0$ in $(0,T)\times M$.
		\end{theorem}
		
		\begin{proof}
			 By the condition, we have that  $u \in L^\infty([0, T], H^2)$. Using the Sobolev embedding, we obtain  $|u|^2u \in L^\infty([0, T], L^2(M))$.
On the slab $(0, T)\times\omega$, we have
\[
(\Delta_g^2-\beta\Delta_g) u =- |u|^2 u -u.
\]
Therefore, $\Delta_g^2 u \in L^2([0, T], L^2(\omega))$ and $u \in L^2([0, T)\times H^4(\omega))$. Since $H^4(\omega)$ is an algebra, we infer that  $u \in C^\infty((0, T)\times\omega)$.

By applying Proposition \ref{nonlinear propagation}, we get $u \in L^2_{loc}((0, T), H^{2+\frac{3-3b}{2}})$. Then taking a time $t_0$ such that $u(t_0) \in H^{2+\frac{3-3b}{2}}$, we can  solve  \eqref{4NLS} in $X^{2+\frac{3-3b}{2}, b}_T$  with initial data $u(t_0)$. By uniqueness in $X^{2, b}_T$, we can conclude that $u \in X^{2+\frac{3-3b}{2}, b}_T$.
By iteration, we get that \(u \in L^2([0, T), H^r)\) for every \(r \in \mathbb{R}\) and \(u \in C^\infty([0, T], M)\).
           Taking the time derivative in both sides of the equation \eqref{lower term u} and  let $v=\partial_{t}u$, we obtain the equation of $v$ as \eqref{equa-v}. Then using the Assumption \ref{UCP-1}, we  obtain $v=\partial_{t}u=0$ on $(0,T)\times M$. Multiplying $\bar{u}$ and integral by parts gives that $u\equiv0$.
		\end{proof}



\section{Observability}\label{sec:Obe}

In this section, we prove the observability for the linear fourth-order Schr\"odinger equation
\begin{equation}\label{nonlinear fourth order equation}
\begin{cases}
iu_{t}+\Delta^{2}_{g}u-\beta\Delta_{g} u=-|u|^{2k}u+h(x,t),&(x,t)\in M\times[0,T],\\ u(0,x)=u_{0}(x),&x\in M.
\end{cases}
\end{equation}
 We proceed by following the strategy of Maci\`a \cite{Macia}.
 
 \begin{proposition}[Linear observabilty]
Let $d\geq1$ and $s\in\R$, suppose that $u$ is a solution to 		
\begin{equation}\label{duality equation}
iu_{t}+\Delta_g^{2}u-\beta\Delta_g u=0, \,\, u(0,x)=u_{0}(x).
\end{equation}
For $T>0$ bounded and open set $\omega$ satisfying Assumption \ref{GCC} , we have
\begin{equation}\label{obervability inequality}
\|u_{0}\|_{L^2}^{2}\leq C\int_{0}^{T}\int_{\omega}|e^{it(\Delta_g^2-\beta\Delta_g)}u_{0}|^2\,\dd x\dd t.
\end{equation}
\end{proposition}
		
\begin{proof}
 By the standard argument of Burq-Zworski \cite{Burq-Zworski}, we only deal with a frequency localized version of \eqref{obervability inequality}.
		
Given $\chi\in C_{c}^{\infty}(\mathbb{R}_{+})$ and $\sigma_{\alpha}(\cdot)\in C_{c}^{\infty}(\mathbb{R}^{+})$ such that $\sigma_{\alpha}(s)=s^{\frac{\alpha}{2}}$ for $s\in\operatorname{supp}\chi$,  we have
\begin{equation}\label{frequency cutoff observability equation}
\begin{cases}
i\partial_{t}u_{h}(t,x)=\sigma_{4}(-h^{2}\Delta_g)u_{h}-\beta\sigma_{2}(-h^{2}\Delta_g) u_{h},& \\u_{h}|_{t=0}=\chi(-h^{2}\Delta_g)u_{0},
\end{cases}
\end{equation}
in which $u_{h}=\chi(-h^2\Delta_g)u$, $u$ is the solution of \eqref{duality equation}.
		
Next, we introduce the Wigner distributions of solutions to \eqref{frequency cutoff observability equation} involving the Hamiltonian. Using a partition of unity on manifolds, we can assume that $a$ is supported in a local chart. In this chart, we define  $\operatorname{Op}_{h}(a)u=\frac{1}{(2\pi h)^{d}}\int_{\mathbb{R}^{2d}}e^{\frac{i}{h}(x-y)\cdot\xi}a(x,\xi)\zeta(y)u(y)\,dyd\xi$, where $\zeta=1$ on a neighborhood of the spatial projection of $\operatorname{supp}(a)$. For function $v$, the Wigner distribution $W_{v}^{h}\in\mathcal{D}'(T^{*}M)$ of $v\in L^{2}(M)$is defined by 
\begin{equation}
\Big\langle W_{v}^{h},a\Big\rangle=(\operatorname{Op}_{h}(a)v,v)_{L^{2}(M)},\forall a\in C_{c}^{\infty}(T^{*}M).
\end{equation}

Let $W_{u_{h}}^{h}(t)\in \mathcal{D}'(T^{*}M)$ denote the Wigner distribution of the function $u_{h}(t,\cdot)$ for some $u_{0}\in L^{2}$. Then for $\forall a\in C_{c}^{\infty}(T^{*}M)$, using \eqref{frequency cutoff observability equation}, we compute
\begin{equation}\label{Wigner} 
\begin{aligned}
&\frac{\dd}{\dd t}\langle W_{u_{h}}^{h}(t),a\rangle\\=&\frac{1}{ih^{4}}\Big([\operatorname{Op}_{h}(a),\sigma_{4}(-h^{2}\Delta_g)]u_{h},u_{h}\Big)_{L^{2}(M)}+\frac{1}{ih^{2}}\big([\operatorname{Op}_{h}(a),-\beta \sigma_{2}(-h^{2}\Delta_g)]u_{h},u_{h}\big)_{L^{2}(M)}\\=&\frac{1}{ih^{4}}(\frac{h}{i}\operatorname{Op}_{h}(\{a,\sigma_{4}(|\xi|^{2}_{g})\})u_{h},u_{h})_{L^{2}(M)}+\frac{1}{ih^{2}}(\frac{h}{i}\operatorname{Op}_{h}(\{a,\beta\sigma_{2}(|\xi|^{2}_{g})\})u_{h},u_{h})_{L^{2}(M)}\\&-\Big( (ih^{-2}\operatorname{Op}_{h}(r_{1})+i\operatorname{Op}_{h}(r_{2}) )u_{h},u_{h}\Big)_{L^{2}(M)}\\=&-\frac{1}{h^{3}}(\operatorname{Op}_{h}(\{a,\sigma_{4}(|\xi|^{2}_{g})\})u_{h},u_{h})_{L^{2}(M)}-\frac{1}{h}(\operatorname{Op}_{h}(\{a,\beta\sigma_{2}(|\xi|^{2}_{g})\})u_{h},u_{h})_{L^{2}(M)}\\&-\Big((ih^{-2}\operatorname{Op}_{h}(r_{1})+i\operatorname{Op}_{h}(r_{2}) )u_{h},u_{h}\Big)_{L^{2}(M)},
\end{aligned}
\end{equation}
in which $r_{1},r_{2}\in S^{0}(T^{*}M)$. We know that $\Big|\big(\operatorname{Op}_{h}(r_{j})u_{h},u_{h}\big)_{L^{2}(M)}\Big|\leq C_{j}\|u_{h}\|^{2}_{L^{2}(M)}$, which implies that $W_{u_{h}}^{h}$ is bounded in  $C(\mathbb{R};\mathcal{D}'(T^{*}M))$. Hence, $W_{u_{h}}^{h}$ has accumulation points in $\mathcal{D}'(\mathbb{R}\times T^{*}M)$, i.e. there exists a semiclassical measure $\mu_{t}$ and a subsequence $W_{h_{\ell}}$ such that
\begin{align}
\int_{\mathbb{R}}\theta(t)\langle W_{h_{\ell}}(t),a\rangle \,\dd t\longrightarrow\int_{\mathbb{R}}\theta(t)\int_{T^{*}M}a(x,\xi)\mu_{t}(\dd x,\dd\xi)\,\dd t, ~~h_{\ell}\to0
\end{align}
for all $\theta\in C_{c}^{\infty}(\mathbb{R})$ and $a\in C_{c}^{\infty}(T^{*}M)$.

After multiplication by $\theta\in C_{c}^{\infty}(\mathbb{R})$ and integral by parts in $t$, the equation \eqref{Wigner}  can be reformulated as  
\begin{align*}
-h_{\ell}^{3}\int_{\mathbb{R}}\theta'(t)\langle W_{h_\ell},a\rangle\,\dd t=&\int_{\mathbb{R}}\theta(t)\big(\operatorname{Op}_{h}(\{a,\sigma_{4}(|\xi|^{2}_{g})\})u_{h_{\ell}},u_{h_{\ell}}\big)_{L^{2}(M)}\,\dd t+\mathcal{O}(h_{\ell})\\=&\int_{\mathbb{R}}\theta(t)\big\langle W^{h_{\ell}}_{u_{h_{\ell}}},\{a,\sigma_{4}(|\xi|^{2}_{g})\}\big\rangle\,\dd t+\mathcal{O}(h_{\ell}).
\end{align*}
Taking $l\to\infty$, we deduce that
\begin{equation}
\int_{\mathbb{R}}\theta(t)\int_{T^{*}M}\{a,\sigma_{4}(|\xi|^{2}_{g})\}\mu_{t}(\dd x,\dd \xi)\,\dd t=0.
\end{equation}
This implies that 
\begin{equation}
\int_{T^{*}M}\{a,\sigma_{4}(|\xi|^{2}_{g})\}\mu_{t}(\dd x,\dd \xi)=0,~~a.e.\,\, t.
\end{equation}

Suppose that $\phi_{t}(x,\xi)$ is the integral flow generated by Hamiltonian vector field of $\sigma_{4}(|\xi|^2_g)$, then 
\begin{equation}
\frac{\dd}{\dd t}\int_{T^{*}M}a(\phi_{t}(x,\xi))\mu_{t}(\dd x,\dd\xi)=\int_{T^{*}M}\{a,\sigma_{4}(|\xi|^{2}_{g})\}\mu_{t}(\dd x,\dd \xi)=0,
\end{equation}
which implies that $\mu_{t}$ is invariant under the Hamiltonian flow generated by $\sigma_{4}(|\xi|_{g}^{2})$. Furthermore, the Hamiltonian vector field of $\sigma_{4}(|\xi|_{g}^{2})$ differs from that of $\sigma_{2}(|\xi|^{2}_{g})$, the two vector fields can be transformed into each other via reparametrization. As a consequence, the orbit of Hamiltonian flow of $\sigma_{4}(|\xi|_{g}^{2})$ (When projected onto the manifold) are geodesics.
		
To finish the proof of \eqref{obervability inequality}, we introduce $\kappa\in C^{\infty}_{c}(\mathbb{R}^{+};[0,1])$ such that $\kappa|_{[1,2]}=1$ and $\kappa(s)=0$ for $s\leq\frac{1}{2}$ and $s\geq\frac{5}{2}$, and denote $\Pi_{h}:=\kappa(-h^{2}\Delta)$. 	It is sufficient to prove a frequency-localized version of \eqref{obervability inequality},
\begin{equation}\label{cutoff observability inequality}
\|\Pi_{h}u_{0}\|^{2}_{L^{2}(M)}\leq C\int_{0}^{T}\int_{\omega}|e^{it(\Delta^{2}_{g}-\beta\Delta_{g})}\Pi_{h}u_0|^2\,\dd x\dd t
\end{equation}
 for any $u_{0}\in L^{2}(M)$ and $0<h<h_{0}$.
		
If \eqref{cutoff observability inequality} is not right, there exist a sequence $h_{j}\to0$ and functions $u_{0,h_{j}}$ such that
\begin{equation}
\|\Pi_{h_{j}}u_{0,h_{j}}\|_{L^{2}(M)}=1, ~~~\lim_{j\to\infty}\int_{0}^{T}\int_{\omega}|e^{it(\Delta_{g}^{2}-\Delta_{g})}\Pi_{h_{j}}u_{0,h_j}|^2\,\dd x\dd t=0.
\end{equation}
The first equality means that $W^{h_{j}}_{u_{h_{j}}}$ converges to a finite measure in $T^{*}M$, and the second inequality means that $W^{h_{j}}_{u_{h_{j}}}$ converges to zero measure in $T^{*}\omega\setminus0$.
Thus, we have 
\begin{equation}\label{contradiction part}
\int_{0}^{T}\mu_{t}(T^{*}M\setminus0)\,\dd t=T,~~\int_{0}^{T}\int_{T^{*}M}b(x)\mu_{t}(\dd x,\dd\xi)\,\dd t=0,~~\forall b\in C^{\infty}_{c}(\omega).
\end{equation}
Denote $$F_{\omega}^{T}:=\bigcup_{t\in[0,T]}\{\phi_{t}(x,\xi):(x,\xi)\in T^{*}\omega\setminus0\}.$$ Since $\omega$ satisfies Geometric Control Condition,  there exsits $T_{0}$ such that $F_{\omega}^{T_{0}}=T^{*}M\setminus0$. Combining the fact that $\mu_{t}$ is invariant by the Hamiltonian flow of $\sigma_{4}(|\xi|_{g}^{2})$, we infer that $\mu_t(T^{*}M\setminus0)=0$, which is a contradiction to the first part of \eqref{contradiction part}.
 \end{proof}



\section{Local Controllability}\label{sec:Loccon}

Before the proof of nonlinear controllability, we briefly recall the HUM method for the equation
\begin{equation}
\begin{cases}
i\partial_{t}u+(\Delta_{g}^{2}-\beta\Delta_{g})u=-|u|^{2k}u+h(x,t),\\ u(0,x)=u_{0}(x).
\end{cases}
\end{equation}
When observability inequality is proved, HUM method gives the exact controllability of linear equation
\begin{equation}\label{linear equation}
\begin{cases}
i\partial_{t}u+(\Delta_{g}^{2}-\beta\Delta_{g})u=0,\\ u(0,x)=u_{0}(x).
\end{cases}
\end{equation}
		
\subsection{Local controllability in energy space}
In this subsection, we prove  Theorem \ref{thm1} for $R_0$ small and $1\leq d\leq 4$. Since $\omega$ satisfies the Geometric Control Condition, we can find $\omega'$ satisfying GCC and $\omega'\Subset \omega$. We then introduce a real-valued function $\varphi\in C^{\infty}(M)$ supported in $\omega$ such that $\varphi=1$ on $\omega'$. Consider the systems
\begin{equation}\label{HUM dual double equation}
\begin{cases}
i\partial_{t}u+\Delta^2_g u-\beta\Delta_{g}u=h\in L^{2}(0,T;H^2), u(T)=0,\\ i\partial_{t}v+\Delta^{2}_{g}v-\beta\Delta_{g}v=0, v(0)=v_{0}\in H^{-2}(M).
\end{cases}
\end{equation}
Multiplying the first equation of \eqref{HUM dual double equation} by $\bar{v}$ and integrating in time and space, we get 
\begin{equation}
-i(u_{0},v_{0})_{L^2(M)}=\int_{0}^{T}(h,v)_{L^{2}(M)}\,\dd t,
\end{equation}
where $u(0)=u_{0}$. We define the continuous map $\Lambda: H^{-2}\to H^{2}$ by $\Lambda v_{0}=-iu_{0}$ with the choice 
$$h=Av=\varphi(1-\Delta_{g})^{-2}\big(\varphi(x) v(x,t)\big).$$ 
In other words, control function $h$ is the regularization of $v$(solution of dual equation) on $\omega'$. This yields
\begin{equation}\label{B}
(\Lambda v_{0},v_{0})_{L^{2}}=\int_{0}^{T}(Av,v)_{L^2}\,\dd t=\int_{0}^{T}\|Bv(t)\|^{2}_{L^2}\,\dd t=\int_{0}^{T}\|\varphi v(t)\|^{2}_{H^{-2}}\,\dd t
\end{equation}
where $Bv(t,x)=(1-\Delta_{g})^{-1}\big(\varphi(x)v(t,x)\big)$. 
		
Hence, the HUM operator $\Lambda$ is selfadjoint and satisfies $\|\Lambda v_{0}\|_{H^2}\geq C\|v_0\|_{H^{-2}}$ by observability inequality \eqref{obervability inequality}. Since $\Lambda$ defines a isomorphism from $H^{-2}$ to $H^2$, we get the controllability of linear equation.

Secondly, for nonlinear equation, we prove the local exact controllability around $u=0$. Precisely, we should establish the following fact: there exist a small number $\delta>0$, such that for all $u_{0}\in B_{\delta}(0)$, one can get a control function $h\in L^{2}([0,T];H^{2})$ such that $u(T,x)=0$.  
		
We consider two systems
\begin{equation}\label{control function equation}
\begin{cases}
i\partial_{t}\Phi+\Delta^{2}_{g}\Phi-\beta\Delta_{g}\Phi=0,\\ \Phi(0)=\Phi_{0}\in H^{-2}
\end{cases}
\end{equation}
and 
\begin{equation}\label{control function equation2}
\begin{cases}
i\partial_{t}u+\Delta_{g}^{2}u-\beta\Delta_{g}u+|u|^{2k}u=A\Phi,\\ u(T)=0.
\end{cases}
\end{equation}
where $A$ is defined by $Av=\varphi(1-\Delta_{g})^{-2}(\varphi v)$.
		
We also define the nonlinear operator
\begin{equation}
\begin{aligned}
&\mathcal{N}:H^{-2}(M)\to H^{2}(M),\\& ~~~~~~~\Phi_{0}\to \mathcal{N}\Phi_{0}=u_{0}.
\end{aligned}
 \end{equation}
The goal is to prove that $\mathcal{N}$ is onto on a small neighborhood of the origin of $H^{2}(M)$. Splitting $u$ into the  linear part $\Psi$ and the nonlinear part $v$, we get $u=v+\Psi$ with 
\begin{equation}\label{nonlinear v equation}
\begin{cases}
i\partial_{t}v+\Delta^2_g v-\beta\Delta_g v+|u|^{2k}u=0,\\ v(T)=0
\end{cases}
\end{equation}
and
\begin{equation}
\begin{cases}
i\partial_{t}\Psi+\Delta^{2}_{g}\Psi-\beta\Delta_{g}\Psi=A\Phi,\\ \Psi(T)=0.
\end{cases}
\end{equation}
From the local wellposedness, we know that $u,v,\Psi\in C([0,T];H^{2}(M))$. Since  $u(0)=v(0)+\Psi(0)$, it holds that
\begin{equation*}
\mathcal{N}\Phi_{0}=K\Phi_0+ S\Phi_0,
\end{equation*}
where $K\Phi_{0}=v(0)$. Observing that $S$ is the linear control isomorphsim between $H^{-2}$ to $H^{2}$, we conclude that $\mathcal{N}\Phi_{0}=u_{0}$ is equivalent to find solution $\Phi_{0}$ to nonlinear equation
\begin{equation}
\Phi_{0}=-S^{-1}K\Phi_0+S^{-1}u_0.
\end{equation}
Defining the operator $B:H^{-2}(M)\to H^{-2}(M)$ by 
\begin{equation*}
B\Phi_{0}=-S^{-1}K\Phi_0+S^{-1}u_0.
\end{equation*}
Then, solving the nonlinear equation is equivalent to find the fixed point of $B$. According to the boundedness of $S^{-1}$,we have 
\begin{equation*}
\|B\Phi_{0}\|_{H^{-2}}\leq C\big(\|K\Phi_{0}\|_{H^{2}}+\|u_{0}\|_{H^{2}}\big)=C\big(\|v(0)\|_{H^{2}}+\|u_{0}\|_{H^{2}}\big).
\end{equation*}
Utilizing the classical energy estimate to system \eqref{nonlinear v equation} in $H^2$, we can get 
\begin{align*}
\|B\Phi_{0}\|_{H^{-2}}\leq& C\Big(\int_{0}^{T}\big\||u|^{2k}u\big\|_{H^{2}}\dd t+\|u_{0}\|_{H^{2}}\Big)\\ \leq& CT^{1-\frac{2k}{p}}\|u\|^{2k+1}_{E_{T}}+\|u_{0}\|_{H^2},
\end{align*}
 where $E_T$ is defined in \eqref{Equation of ET}. Taking $\|\Phi_0\|_{H^{-2}}\leq R$ with $R$ small enough, we have \begin{equation}
\|u\|_{E_{T}}\leq C\sup_{t\in[0,T]}\|A\Phi\|_{H^{2}}\leq C\|\Phi_{0}\|_{H^{-2}}.
\end{equation}
Thus
\begin{align*}
\|B\Phi_{0}\|_{H^{-2}}\leq CT^{1-\frac{2k}{p}}\|\Phi_{0}\|^{2k+1}_{H^{-2}}+C\|u_{0}\|_{H^{2}}.
\end{align*}
Choosing $\|u_{0}\|_{H^2}\leq\frac{R}{2C}$ and $R\leq(\frac{1}{2CT^{1-\frac{2k}{p}}})^{\frac{1}{2k}} $, we can get $\|B\Phi_0\|_{H^{-2}}\leq R$. Next, we should show that $B$ is a contraction map.
		
For systems
\begin{equation}
\begin{cases}
i\partial_{t}(v_{1}-v_{2})+\Delta_{g}^{2}(v_{1}-v_{2})-\beta\Delta_{g}(v_{1}-v_{2})+|u_{1}|^{2k}u_{1}-|u_{2}|^{2k}u_{2}=0,\\ (v_{1}-v_{2})(T)=0
\end{cases}
\end{equation}
and 
\begin{equation}
\begin{cases}
i\partial_{t}(u_{1}-u_{2})+\Delta_{g}^{2}(u_{1}-u_{2})-\beta\Delta_{g}(u_{1}-u_{2})+|u_{1}|^{2k}u_{1}-|u_{2}|^{2k}u_{2}=A(\Phi_{1}-\Phi_{2}), \\(u_{1}-u_{2})(T)=0.
\end{cases}
\end{equation}
Since
\begin{align*}
\|B\Phi_{0}^{1}-B\Phi_{0}^{2}\|_{H^{-2}}\leq& C\|(v_1-v_2)(0)\|_{H^2}\\ \leq& C\Big(\|u_{1}\|^{2k}_{E_{T}}+\|u_{2}\|^{2k}_{E_{T}}\Big)\|u_{1}-u_{2}\|_{E_{T}}.
\end{align*}
Hence
\begin{align}
\|B\Phi_{0}^{1}-B\Phi_{0}^{2}\|_{H^{-2}}\leq 2CR^{2k}\|u_{1}-u_{2}\|_{E_{T}}.
\end{align}
		
Next, we estimate $\|u_{1}-u_{2}\|_{E_{T}}$,
\begin{align*}
\|u_{1}-u_{2}\|_{E_T}\leq& C\Big(\|u_{1}\|_{E_T}^{2k}+\|u_{2}\|_{E_T}^{2k}\Big)\|u_{1}-u_{2}\|_{E_{T}}+C\sup_{0\leq t\leq T}\|A(\Phi^{1}-\Phi^{2})\|_{H^{2}}\\\leq& CR^{2k}\|u_{1}-u_{2}\|_{E_{T}}+C\sup_{0\leq t\leq T}\|A(\Phi^{1}-\Phi^{2})\|_{H^{2}}.
\end{align*}
Taking $R$ such that $CR^{2k}\leq\frac{1}{2}$, we get
\begin{equation}
\|u_{1}-u_{2}\|_{E_{T}}\leq C\sup_{0\leq t\leq T}\|A(\Phi^1-\Phi^2)\|_{H^2}.
\end{equation}
Then, 
\begin{align*}
\|B\Phi_{0}^{1}-B\Phi_{0}^{2}\|_{H^{-2}}\leq& 2CR^{2k}\|u_{1}-u_{2}\|_{E_{T}}\\ \leq& 2CR^{2k}\sup_{0\leq t\leq T}\|A(\Phi^1-\Phi^2)\|_{H^2}\\ \leq& 2CR^{2k}\|\Phi_{0}^{1}-\Phi_{0}^{2}\|_{H^{-2}}.
\end{align*}
Taking $R\leq\delta= \min\{\frac{1}{(4C)^{\frac{1}{2k}}},(\frac{1}{2CT^{1-\frac{2k}{p}}})^{\frac{1}{2k}}\}$, $B$ is a contraction map on $B_{\delta}(0)\in H^{-2}(M)$. Thus, we find the control function $A\Phi$, in which $\Phi$ is solution of \eqref{control function equation} with initial data $\Phi_0$.
		
\subsection{Local controllability in Bourgain space}
Similar to the  case $d\leq 4$, we consider two systems
\begin{equation}\label{Bourgain control1}
\begin{cases}
i\partial_{t}\Phi+\Delta^{2}_{g}\Phi-\beta\Delta_{g}\Phi=0,&x\in\Bbb S^5\\ \Phi(0)=\Phi_{0}\in H^{-2}
\end{cases}
\end{equation}
and 
\begin{equation}\label{Bourgain control2}
			\begin{cases}
				i\partial_{t}u+\Delta_{g}^{2}u-\beta\Delta_{g}u+|u|^{2}u=A\Phi,\\ u(T)=0
			\end{cases}
\end{equation}
where $A$ is defined by $Av=\varphi(1-\Delta_{g})^{-2}(\varphi v)$. We also define the operator
\begin{equation}
\begin{aligned}
			&\mathcal{N}:H^{-2}(\Bbb S^5)\to H^{2}(\Bbb S^5),\\& ~~~~~~~\Phi_{0}\to \mathcal{N}\Phi_{0}=u_{0}.
\end{aligned}
 \end{equation}
The goal is to prove that $\mathcal{N}$ is onto on a small neighborhood of the origin of $H^{2}(\mathbb{S}^{5})$. Splitting $u$ into the  linear part $\Psi$ and the nonlinear part $v$ satisfying $u=v+\Psi$ with 
\begin{equation}\label{split nonlinear part}
\begin{cases}
i\partial_{t}v+\Delta_{g}^{2}v-\beta\Delta_{g}v+|u|^{2}u=0,\\ v(T)=0
\end{cases}
\end{equation}
and 
\begin{equation}
\begin{cases}
i\partial_{t}\Psi+\Delta^{2}_{g}\Psi-\beta\Delta_{g}\Psi=A\Phi,\\ \Psi(T)=0.
\end{cases}
\end{equation}
From the wellposedness, we know that $u,v,\Psi\in X_{T}^{s,b}$ and $u(0)=v(0)+\Psi(0)$ which can be written as 
\begin{equation*}
\mathcal{N}\Phi_{0}=K\Phi_{0}+S\Phi_{0}
\end{equation*}
where $K\Phi_{0}=v(0)$. Observing that $S$ is the linear  isomorphsim between $H^{-2}$ to $H^{2}$, we conclude that $\mathcal{N}\Phi_{0}=u_{0}$ is equivalent to find solution $\Phi_{0}$ to nonlinear equation
\begin{equation}
\Phi_{0}=-S^{-1}K\Phi_0+S^{-1}u_0.
\end{equation}
		
We denote the operator by $B:H^{-2}(\Bbb S^5)\to H^{-2}(\Bbb S^5)$. More precisely,   
\begin{equation*}
B\Phi_{0}=-S^{-1}K\Phi_0+S^{-1}u_0.
\end{equation*}
The problem now is to find a fixed point near the origin in $H^{-2}(\mathbb{S}^{5})$. We will prove that $B$ is a contraction map on a small ball $B_{H^{-2}}(0,\delta)$ if $\|u_{0}\|_{H^{2}}$ is small enough. We know that 
\begin{equation}\label{shit}
\|B\Phi_{0}\|_{H^{-2}}\leq C(\|K\Phi_{0}\|_{H^{2}}+\|u_{0}\|_{H^{2}})=C(\|v(0)\|_{H^{2}}+\|u_{0}\|_{H^{2}}).
\end{equation}
It remains to estimate $\|v(0)\|_{H^{2}}$. Applying the estimate for system \eqref{split nonlinear part} in $X^{s,b}$ space, we have
\begin{align*}
			\|v(0)\|_{H^{2}}\leq& C\|v\|_{X^{2,b}_{T}}\leq C\||u|^{2}u\|_{X^{2,-b'}_{T}} \leq C\|u\|_{X^{2,b}_{T}}^{3}.
		\end{align*}
Notice that for $\eta\ll1$ small enough, it holds
		\begin{equation*}
			\|A\Phi\|_{L^{2}([0,T],H^{2})}\leq C\|\Phi\|_{X_{T}^{-2,b}}\leq C\|\Phi_{0}\|_{H^{-2}}\leq\eta,
		\end{equation*}
then we have
		\begin{equation*}
			\|u\|_{X^{2,b}_{T}}\leq c\|\Phi_{0}\|_{H^{-2}}.
		\end{equation*}
From \eqref{shit}, we get
		\begin{equation*}
			\|B\Phi_{0}\|_{H^{-2}}\leq C(\eta^{3}+\|u_{0}\|_{H^{2}}).
		\end{equation*}
Taking $\|u_{0}\|_{H^{2}}\leq\frac{\eta}{2C}$ and let $C\eta^{3}\leq\frac{\eta}{2}$, we obtain $\|B\Phi_{0}\|_{H^{-2}}\leq\eta$. Next, we prove that $B$ is a contraction mapping. For systems
\begin{equation}
\begin{cases}
i\partial_{t}(v_{1}-v_{2})+\Delta_{g}^{2}(v_{1}-v_{2})-\beta\Delta_{g}(v_{1}-v_{2})+|u_{1}|^{2}u_{1}-|u_{2}|^{2}u_{2}=0,\\ (v_{1}-v_{2})(T)=0
\end{cases}
\end{equation}
and 
\begin{equation}
\begin{cases}
i\partial_{t}(u_{1}-u_{2})+\Delta_{g}^{2}(u_{1}-u_{2})-\beta\Delta_{g}(u_{1}-u_{2})+|u_{1}|^{2}u_{1}-|u_{2}|^{2}u_{2}=A(\Phi_{1}-\Phi_{2}), \\(u_{1}-u_{2})(T)=0.
			\end{cases}
\end{equation}
We obtain 
		\begin{align*}
			\|B\Phi^{1}_{0}-B\Phi^{2}_{0}\|_{H^{-2}}\leq &C\|(v_{1}-v_{2})(0)\|_{H^{2}}\leq C\||u_{1}|^{2}u_{1}-|u_{2}|^{2}u_{2}\|_{X^{2,-b'}_{T}}\\ \leq &C\big(\|u_{1}\|_{X^{2,b}_{T}}^{2}+\|u_{2}\|_{X^{2,b}_{T}}^{2}\big)\|u_{1}-u_{2}\|_{X^{2,b}_{T}}\\ 
			\leq& C\eta^{2}\|u_{1}-u_{2}\|_{X^{2,b}_{T}}.
		\end{align*}
Next, we estimate $\|u_{1}-u_{2}\|_{X^{2,b}_{T}}$, 
		\begin{align*}
			\|u_{1}-u_{2}\|_{X^{2,b}_{T}}\leq &C\big(\|u_{1}\|_{X_{T}^{2,b}}^{2}+\|u_{2}\|_{X_{T}^{2,b}}^{2}\big)\|u_{1}-u_{2}\|_{X_{T}^{2,b}}+C\|A(\Phi^{1}-\Phi^{2})\|_{L^{2}([0,T],H^{2})}\\ \leq& C\eta^{2}\|u_{1}-u_{2}\|_{X_{T}^{2,b}}+C\|\Phi^{1}_{0}-\Phi^{2}_{0}\|_{H^{-2}}.
		\end{align*}
If $\eta$ small, we get 
		\begin{equation*}
			\|u_{1}-u_{2}\|_{X^{2,b}_{T}}\leq C\|\Phi^{1}_{0}-\Phi^{2}_{0}\|_{H^{-2}}.
		\end{equation*}
Then 
		\begin{equation*}
			\|B\Phi_{0}^{1}-B\Phi_{0}^{2}\|_{H^{-2}}\leq C\eta^{2}\|\Phi^{1}_{0}-\Phi^{2}_{0}\|_{H^{-2}}.
		\end{equation*}
Let $\eta$ satisfy $C\eta^{2}\leq\frac{1}{2}$ and we get the contraction map.
		


\section{Stabilization}\label{sec:Sta}
		
In this section, we prove the stabilization result for \eqref{4NLS} in energy spaces for $d\leq4$. For $d=5$, we will work on the Bourgain space to obtain the stabilization result in $H^2$ space. Once we obtain the global stabilization, we can extend local controllability to semi-global controllability. 
		
\subsection{Stabilization in energy space}\label{stabilization energy}
		
For convenience, we recall \eqref{damping-NLS} here,
\begin{align}\label{damping-NLS2}
			\begin{cases}
				i\partial_tu+(\Delta_g^2-\beta\Delta_g)u=-|u|^{2k}u-u+a(x)(1-\Delta_g)^{-2}a(x)\partial_tu,&(t,x)\in\R\times M,\\
				u(0,x)=u_0(x).
\end{cases}
\end{align}
The energy $E(t)$ to this equation is defined by \eqref{nonlinear energy}.
		The aim in this subsection is to show that for arbitrary $T>0$, $R_{0}>0$, there exists a $C(T,R_{0})$ such that
		\begin{equation}\label{stabilization inequality}
			E(0)\leq C\int_{0}^{T}\|(1-\Delta_{g})^{-1}\big(a(x)\partial_{t}u\big)\|^{2}_{L^{2}}\,\dd t
		\end{equation}
		holds for any $u_{0}\in H^{2}(M)$ satisfying $\|u_{0}\|_{H^{2}(M)}\leq R_{0}$.
		Combining with $$E(T)-E(0)=-\int_{0}^{T}\|(1-\Delta_{g})^{-1}\big(a(x)\partial_{t}u\big)\|^{2}_{L^{2}}\,\dd t,$$
		From \eqref{stabilization inequality}, we get 
		\begin{equation}
			E(T)-E(0)\leq -\frac{1}{C}E(0).
		\end{equation}
		Then, it holds that
		\begin{equation}
			E(T)\leq(1-\frac{1}{C})E(0).
		\end{equation}
		By iteration, we can get 
		\begin{equation*}
			E(nT)\leq(1-\frac{1}{C})^{n}E(0)\Longrightarrow E(t)\leq e^{-\gamma t}E(0),
		\end{equation*}
		which is exponential decay stabilization, which implies Theorem \ref{thm2} for $1\leq d\leq4$.
		
		To prove \eqref{stabilization inequality}, we argue by contradiction. We suppose  that there exists a sequence $\{u_{\ell}\}$ of solutions to damping Schr\"odinger equation satisfying $$\|u_{\ell}(0)\|_{H^{2}}\leq R_{0}$$ and $$\int_{0}^{T}\|(1-\Delta_g)^{-1}\big(a(x)\partial_{t}u_{\ell}\big)\|_{L^{2}}^{2}\,\dd t\leq\frac{1}{\ell}E(u_{\ell}(0)).$$ Denoting $\alpha_{\ell}=(E(u_{\ell}(0)))^{\frac{1}{2}}$, due to the boundedness of $u_{\ell}(0)$ in $H^{2}$, we infer that there exist a subsequence of $\alpha_{\ell}$ converges to $\alpha\geq0$. 
		
		Case 1: $\alpha>0$. By conservation of energy,  $\{u_{\ell}\}$ is bounded in $L^{\infty}([0,T]; H^{2}(M))$. Moreover, using the structure of damped Schr\"odinger equation
		\begin{align*}
			(i-a(x)(1-\Delta_{g})^{-2}a(x))\partial_{t}u=-(\Delta_{g}^{2}-\beta\Delta_{g})u-|u|^{2k}u-u\\\Longrightarrow \partial_{t}u=J^{-1}\Big(-(\Delta_{g}^{2}-\beta\Delta_{g})u-|u|^{2k}u-u\Big),
		\end{align*}
		we observe that $u_{\ell}\in C^{1}([0,T];\mathcal{D}'(M))$. Thus we find a subsequence of $u_{\ell}$, still denoted by $u_{\ell}$, such that $\forall t\in[0,T]$, $u_{\ell}(t)\rightharpoonup u(t)$ for some $u\in C_{w}([0,T];H^{2}(M))$. Passing to the limit in the damped Schr\"odinger system \eqref{damping-NLS2} for $u_{\ell}$, we get 
		\begin{equation}\label{nonlinear unique continuation equation}
			\begin{cases}
				i\partial_{t}u+\Delta_{g}^{2}u-\beta\Delta_{g}u=-|u|^{2k}u, &(x,t)\in M\times(0,T),\\ \partial_{t}u=0,&(x,t)\in\omega\times(0,T).
			\end{cases}
		\end{equation}
		
		If $u\in L^{\infty}([0,T];H^{2}(M))$ is a solution of the equation \eqref{nonlinear unique continuation equation} then $u\in L^{p}\big([0,T]; L^{\infty}(M)\big)$ for every $p<\infty$, and the nonlinear term $|u|^{\alpha-1}u\in L^{2}([0,T];H^{2}(M))$.     By Duhamel formula and the uniquness of the Cauchy problem for fourth order Schr\"odinger equation we conclude that $u\in C([0,T];H^{2}(M))$. Then by Assumption \ref{UCP-1},
 we get $u=0$. Hence, the sequence  converges to $0$ weakly.  As a consequence, we obtain that $|u_{\ell}|^{2k}u_{\ell}$ converges to $0$ in $L^{2}([0,T];H^{2})$ strongly. On the other hand, the damping term also converges to $0$ in $L^{2}([0,T];H^{2})$. Up to now, we have 
		\begin{equation}
			\begin{cases}
				\sup_{t\in[0,T]}\|u_{\ell_{j}}(t)\|_{H^{2}}\leq C,\\ \sup_{t\in[0,T]}\|u_{\ell_{j}}\|_{L^{2}}\to 0,\\ \|i\partial_{t}u_{\ell_{j}}+\Delta_{g}^{2}u_{\ell_{j}}-\beta\Delta_{g}u_{\ell_{j}}\|_{L^{2}([0,T];H^{2}(M))}\to 0
			\end{cases}
		\end{equation}
		Thus, we infer that $\|u_\ell\|_{L^{\infty}([0,T];H^{2}(M))}\to0$, which contradicts $\alpha>0$.
		
		Case 2: $\alpha=0$. Set $v_{\ell}=\frac{u_{\ell}}{\alpha_{\ell}}$, and then $v_{\ell}$ satisfies
		\begin{equation}\label{equation v}
			i\partial_{t}v_{\ell}+\Delta_{g}^{2}v_{\ell}-\Delta_{g}v_{\ell}=-|\alpha_{\ell}v_{\ell}|^{2k}v_{\ell}+a(x)(1-\Delta_{g})^{2}a(x)\partial_{t}v_{\ell}=0,(x,t)\in M\times(0,T)
		\end{equation}
		and the estimate
		\begin{equation*}
			\int_{0}^{T}\|(1-\Delta_{g})^{-2}a(x)\partial_{t}v_{\ell}\|_{L^{2}(M)}^{2}\,\dd t\leq \frac{1}{\ell}
		\end{equation*}
		with $\|v_{\ell}(0)\|_{H^{2}(M)}\sim1$. Since $\{v_{\ell}\}$ is bounded in $L^{\infty}(0,T;H^{2}(M))$ and $C^{1}([0,T];\mathcal{D}'(M))$, it admits a subsequence, still denoted by $v_{\ell}$, such that for every $t\in[0,T]$, $v_{\ell}(t)\rightharpoonup v(t)$ for some $v\in C_{w}([0,T];H^{2}(M))$. In addition, applying the estimate to equation \eqref{equation v}, we get
		\begin{equation}
			\|v_{\ell}\|_{E_{T}}\leq C(1+\alpha_{\ell}^{2k}\|v_{\ell}\|^{2k+1}_{E_{T}}).
		\end{equation}
		Since $\|v_{\ell}\|_{E_{T}}$ depends continuously on $T$ and is bounded for $T=0$, we obtain by bootstrap argument that $\|v_{\ell}\|_{E_{T}}$ is bounded and therefore $|v_{\ell}|^{2k}v_{\ell}$ converges to $0$ in $L^{2}([0,T];H^{2}(M))$. Thus, the weak limit satisfies 
		\begin{equation*}
			\begin{cases}
				i\partial_{t}v+\Delta_{g}^{2}v-\beta\Delta_{g}v=0, &(x,t)\in M\times(0,T),\\\partial_{t}v=0, & (x,t)\in \omega\times(0,T).
			\end{cases}
		\end{equation*}
		By the uniqueness of the Cauchy problem for the linear Schr\"odinger equation, we get  $v\in C\big([0,T];H^{2}(M)\big)$ and therefore it is identically to $0$ by unique continuation (observability inequality \eqref{obervability inequality}). Thus, we obtain $v_{\ell}\rightharpoonup 0$. Arguing as in the case 1, we can also get $v_{\ell}$ strongly converges to $0$ in $C([0,T];H^{2}(M))$ and therefore contradicts to  $\|v_{\ell}(0)\|_{H^{2}}=1$. Together with the above results, we get global exact controllability in a large time $T$ which depends on $R_{0}$, $\omega$.
		
\subsection{Stabilization in Bourgain space}
In this part, we study the energy decay rate of  
\begin{equation}\label{cubic bourgain}
i\partial_{t}u+(\Delta_{g}^{2}-\beta\Delta_{g})u+(1+|u|^{2})u=-a(x)(1-\Delta_{g})^{-2}\big(a(x)\partial_{t}u\big).
\end{equation}
in dimension $5$. By the argument in subsection \ref{stabilization energy}, we will prove that for every $T>0$ and every $R_{0}$, there exists a constant $C>0$ such that 
		\begin{equation}
			E(0)\leq C\int_{0}^{T}\big\|(1-\Delta_{g})^{-1}\big(a(x)\partial_{t}u\big)\big\|_{L^{2}}^{2}\,\dd t.
		\end{equation}
		
We argue by contradiction. Suppose that the existence of a sequence $(u_{n})$ of solution of \eqref{cubic bourgain} such that 
		\begin{equation*}
			\|u_{n}(0)\|_{H^{2}}\leq R_{0}
		\end{equation*}
		and 
		\begin{equation}\label{contr observability}
			\int_{0}^{T}\big\|(1-\Delta_{g})^{-1}(a(x)\partial_{t}u_{n})\big\|^{2}_{L^{2}}\,\dd t\leq\frac{1}{n}E\big(u_{n}(0)\big).
		\end{equation}
Denote $\alpha_{n}=E\big(u_{n}(0)\big)^{\frac{1}{2}}$, by Sobolev embedding, we know that $0\leq\alpha_{n}\leq C(R_{0})$. Hence, up to extraction, we can suppose that $\alpha_{n}\to\alpha$, where $\alpha=0$ or $\alpha>0$.
		
Case $1$,  $\alpha>0$. According to the dissipation of energy, $(u_{n})$ is bounded in $L^{\infty}([0,T],H^{2})$ and has a bounded energy. Then by the local theory, we have  $u_n\in X^{2,b}_{T}$. Since $X^{2,b}_{T}$ is a Hilbert space, we can extract a subsequence such that $u_{n}\rightharpoonup u$ weakly in $X^{2,b}_{T}$ and strongly in $X^{s,b'}_{T}$ for one $u$ in $X^{2,b}_{T}$ and $s>\frac{3}{2}$. Therefore, $|u_{n}|^{2}u_{n}$ converges to $|u|^{2}u$ in $X^{s,-b'}_{T}$. Using \eqref{contr observability} and passing to the limit in the equation, we get 
\begin{equation*}
\begin{cases}
i\partial_{t}u+\Delta^{2}_{g}u-\beta\Delta_{g}u+|u|^{2}u+u=0,&(t,x)\in[0,T]\times\mathbb{S}^{5}\\ \partial_{t}u=0,&(t,x)\in(0,T)\times\omega.
			\end{cases}
\end{equation*}
By Theorem \ref{five dimension} we get $u\equiv0$. Therefore, we have $u_{n}(0)\rightharpoonup0$ in $H^{2}$. Using Proposition \ref{linearization}, we infer that $|u_{n}|^{2}u_{n}\to0$ in $X^{2,-b'}_{T}$. Moreover, by \eqref{contr observability} we have 
		\begin{equation*}
			a(x)(1-\Delta_{g})^{-2}(a(x)\partial_{t}u_{n})\to0 ~\textit{in}~L^{2}([0,T],H^{2})
		\end{equation*}
		and the convergence also holds in $X_{T}^{2,-b'}$. This implies $a(x)\partial_{t}u_{n}\to0$ in $L^{2}([0,T],H^{-2})$. According to the equation \eqref{cubic bourgain}, we obtain 
\begin{equation*}
a(x)\big[\Delta_{g}^{2}u_{n}-\beta\Delta_{g}u_{n}+(1+|u_{n}|^{2})u_{n}+a(x)(1-\Delta_{g})^{-2}(a(x)\partial_{t}u_{n})\big]\to0,
\end{equation*}
By Sobolev embedding, we have $u_{n}$ tends to $0$ in $L^{\infty}([0,T],L^{p})$ for any $p<10$. Thus, $|u_{n}|^{2}u_{n}$ converges to $0$ in $L^{\infty}([0,T],L^{q})$ for $q<\frac{10}{3}$ and so in $L^{2}([0,T],H^{-2})$. We get 
\begin{equation*}
a(x)[\Delta_{g}^{2}-\beta\Delta_{g}+1]u_{n}\to 0,~\textit{in}~L^{2}([0,T],H^{-2}).
\end{equation*}
Observing that
\begin{align*}
(1-\beta\Delta_{g}+\Delta_{g}^{2})^{1/2}(a(x)u_{n})=&(1-\beta\Delta_{g}+\Delta_{g}^{2})^{-1/2}(1-\beta\Delta_{g}+\Delta_{g}^{2})(a(x)u_{n})\\ =&(1-\beta\Delta_{g}+\Delta_{g}^{2})^{-1/2}a(x)(1-\beta\Delta_{g}+\Delta_{g}^{2})u_{n}\\&+(1-\beta\Delta_{g}+\Delta_{g}^{2})^{-1/2}[1-\beta\Delta_{g}+\Delta_{g}^{2},a(x)]u_{n},
\end{align*}
this quantity converges to $0$ in $L^{2}([0,T],L^{2})$. Then we have 
\begin{equation*}
\begin{cases}
u_{n}\rightharpoonup0, ~\textit{in}~X_{T}^{2,b'},\\ a(x)u_{n}\to 0~\textit{in}~L^{2}([0,T],H^{2}),\\ i\partial_{t}u_{n}+\Delta^{2}_{g}u_{n}-\beta\Delta_{g}u_{n}\to0~\textit{in}~X^{2,-b'}_{T}.
\end{cases}
\end{equation*}
Thus, combining with  $1-b'>\frac{1}{2}$, it yields
\begin{equation*}
u_{n}(0)\to0, \text{ in } H^{2}. 
\end{equation*}
This means that $E(u_{n}(0))\to 0$, which is a contradiction to our hypothesis $\alpha>0$.
		
Case 2: $\alpha=0$, let us denote $v_{n}=\frac{u_{n}}{\alpha_{n}}$ , which is the solution of the system 
\begin{equation*}
i\partial_{t}v_{n}+(\Delta_{g}^{2}-\beta\Delta_{g})v_{n}+|\alpha_{n}v_{n}|^{2}v_{n}+v_{n}=-a(x)(1-\Delta_{g})^{-2}\big(a(x)\partial_{t}v_{n}\big)
\end{equation*}
and 
\begin{equation}
\int_{0}^{T}\big\|(1-\Delta_{g})^{-1}(a(x)\partial_{t}v_{n})\big\|^{2}_{L^{2}}\,\dd t\leq\frac{1}{n}.
\end{equation}
In this case,
\begin{equation*}
\|v_{n}(t)\|_{H^{2}}=\frac{\|u_{n}(t)\|_{H^{2}}}{E(u_{n}(0))^{1/2}}\leq C\frac{E(u_{n}(t))^{1/2}}{E(u_{n}(0))^{1/2}}\leq C
\end{equation*}
and 
\begin{equation}
\|v(0)\|_{H^{2}}=\frac{\|u_{n}(0)\|_{H^{2}}}{E(u_{n}(0))^{1/2}}\geq  c>0,
\end{equation}
and this means that $\|v_{n}(0)\|_{H^{2}}\sim1$ and $v_{n}$ is bounded in $L^{\infty}([0,T],H^{2})$. By the following nonlinear estimate 
\begin{equation*}
\|v_{n}\|_{X^{2,b}_{T}}\leq C\|v_{n}(0)\|_{H^{2}}+CT^{1-b-b'}\big(\|v_{n}\|_{X^{2,b}_{T}}+\alpha_{n}^{2}\|v_{n}\|^{3}_{X^{2,b}_{T}}\big)
\end{equation*}
and	taking $T$ such that $CT^{1-b-b'}\leq\frac{1}{2}$, we have 
\begin{equation*}
\|v_{n}\|_{X^{2,b}_{T}}\leq C(1+\alpha_{n}^{2}\|v_{n}\|_{X_{T}^{2,b}}^{3}).
\end{equation*}
By bootstrap argument, we get $\|v_{n}\|_{X^{2,b}_{T}}$ is uniformly bounded. Using  Lemma \ref{covering}, we know that $v_{n}$ is bounded in $X^{2,b}_{T}$ for larger $T$, and thus $\alpha_{n}^{2}|v_{n}|^{2}v_{n}$ tends to $0$ in $X^{2,-b'}_{T}$.   Then, we can extract a subsequence such that $v_{n}\rightharpoonup v$ in $X^{2,b}_{T}$ where $v$ is solution of 
\begin{equation*}
\begin{cases}
i\partial_{t}v+(\Delta_{g}^{2}-\Delta_{g})v+v=0,&(t,x)\in(0,T)\times\mathbb{S}^{5},\\ \partial_{t}v=0,&(t,x)\in(0,T)\times\omega.
\end{cases}
\end{equation*}
According to the Assumption \ref{UCP-1}, $v\equiv0$. Then we  get $v_{n}\to0 $ in $X^{2,b}_{T}$, which contradicts to $\|v_{n}(0)\|_{H^{2}}\sim1$.

\end{document}